\documentclass{article}
\usepackage[utf8]{inputenc}
\usepackage{fullpage}
\usepackage{hyperref}
\usepackage{xcolor}
\usepackage{amsmath,amsfonts,amssymb}
\usepackage{graphicx}%
\usepackage{subcaption}
\usepackage{mathrsfs}
\usepackage[margin=1.0in]{geometry}
\usepackage{appendix}
\usepackage{accents}

\usepackage{accents}

\usepackage{amsthm}
\usepackage[shortlabels]{enumitem}
\newtheorem{theorem}{Theorem}[section]

\newtheorem{proposition}[theorem]{Proposition}

\newtheorem{lemma}[theorem]{Lemma}
\newtheorem{corollary}[theorem]{Corollary}
\newcommand{\an}[1]{{\color{orange}{#1}}}

\newcommand{\sfR}{{\mathsf R}}

\newcommand{\sk}{{\rm k}}

\newcommand{\eps}{\varepsilon}
\newcommand{\E}{\mathbb E}
\newcommand{\Rm}{\mathbb R}

\newcommand{\sL}{{\mathsf L}}

\title{Complex Gaussianity and spatio-frequential memory effect of random wave processes}
\author{Guillaume Bal \thanks{Departments of Statistics and Mathematics and Committee on Computational and Applied Mathematics, University of Chicago, Chicago, IL 60637; guillaumebal@uchicago.edu} \and Anjali Nair \thanks{Department of Statistics and Committee on Computational and Applied Mathematics, University of Chicago, Chicago, IL 60637; anjalinair@uchicago.edu}}
\date{\today}

\begin{document}

\maketitle

\begin{abstract}
    Wavefield speckle patterns are generated by interference of randomly scattered coherent light. In the weak-coupling regime of the It\^o-Schr\"odinger paraxial model for long-distance wave propagation, we show the following multiscale character: a macroscopic envelope solves a deterministic diffusion equation while the local wavefield (the speckle) is described by a complex Gaussian process both in terms of spatial axial and lateral displacements as well as frequency and angular variations of the incident wavebeam. These results describe speckle patterns and corroborate chromato-spatial memory effects observed in laser light propagation through heterogeneous media.
\end{abstract}


\section{Introduction} \label{sec:intro}

Our starting point is the following It\^o-Schr\"{o}dinger scalar model of wavebeam propagation in random media:
\begin{equation}\label{eqn:IS_original}
    \mathrm{d}u=\frac{i}{2\omega}\Delta_xu\mathrm{d}z-\frac{\omega^2R(0)}{8}u\mathrm{d}z+\frac{i\omega}{2}u\mathrm{d}B,\quad u(z=0,x)=u_0(x),\quad (z,x)\in[0,\infty)\times\mathbb{R}^d.
\end{equation}
Here, $u_0(x)$ is an incident wavebeam on a hyperplane $z=0$ generating the solution $u(z,x)$ for axial coordinate $z>0$ and transverse coordinates $x\in\Rm^d$ for $d\geq1$; $\omega\in\Rm$ is the wave frequency; $\Delta_x$ is the standard Laplacian in transverse variables; and $B$ is a real valued mean zero Gaussian random field with correlation 
\begin{equation}\nonumber
    \mathbb{E}B(z_1,x_1)B(z_2,x_2)=\min(z_1,z_2)R(x_1-x_2),
\end{equation}
for $R(x)$ a smooth, real-valued function. The product $udB$ should be interpreted in the It\^o sense. We refer to \cite{dawson1984random} for an analysis of such a random wave process model and to \cite{garnier2009coupled,bailly1996parabolic} for a derivation of the model from a scalar Helmholtz equation.

Our objective is to analyze the long-distance behavior of $u(z,x)$ in terms of the spatial variables $(z,x)$ as well as the frequency $\omega$ and orientation of the incident beam. Orientation variations are modeled by replacing $u_0(x)$ by $e^{i\sk\cdot x} u_0(x)$. More precisely, writing the wavefield as $u(z+h,r+x;\omega_0+\Omega,k_0+\kappa)$, we wish to show that the behavior in $(z,r)$ is characterized by a $(\omega_0,k_0)$-dependent solution of a diffusion equation while $u$ is (asymptotically) a mean zero complex Gaussian field in $(h,x,\Omega,\kappa)$ modeling speckle and memory effect as these parameters vary \cite{osnabrugge2017generalized,zhu2020chromato,garnier2023speckle}. Such results are obtained as a long-distance limit in the so-called diffusion scaling of the weak-coupling regime. They generalize those in \cite{bal2024complex} showing the Gaussianity of the process in $x$ for vanishing $(h,\Omega,\kappa)$. 

Complex Gaussian distributions are entirely characterized by their first and second statistical moments, with the real and imaginary parts of the field obeying i.i.d. Gaussian distributions~\cite{reed1962moment}. Such distributions heuristically describe high frequency wave propagation through random media over long distances~\cite{valley1976application, yakushkin1978moments, ishimaru1978wave, sheng1990scattering}. This also forms a model for speckle patterns, observed in physical experiments~\cite{goodman1976some, carminati2021principles}. While fairly well-accepted in the physical literature, due to the highly complex nature of wave propagation through random media~\cite{fouque2007wave, garnier2018multiscale}, such models do not have any rigorous justification in general. However, the paraxial approximation of the Helmholtz equation provides a more amenable framework for mathematical analysis. This ignores backscattering, and is routinely used to model laser propagation through optical turbulence~\cite{andrews2001laser}. Justification of such models starting from the scalar Helmholtz equation and the passage to the It\^o-Schr\"{o}dinger equation~\eqref{eqn:IS_original} can be found in~\cite{bailly1996parabolic, garnier2009coupled}. The It\^o-Schr\"{o}dinger equation will be our starting point for this paper as well.

As in \cite{bal2024complex}, the appropriate diffusive scaling of the weak-coupling regime is obtained for the following rescaled axial variable and choice of coupling strength:
%We consider the scintillation scaling of the It\^o-Schr\"{o}ding equation
\begin{equation}\label{eqn:scint_scaling}
    z\to \frac{\eta z}{\eps},\quad R^\eps(x)=\frac{\eps}{\eta^3}R(x),
\end{equation}
where $0<\eps\ll \eta\ll 1$ with a separation of scales $\eta=(\log|\log\eps|)^{-1}$ for technical reasons. The optical depth of the medium (for $z\approx 1$) is given by $\frac{\eta }{\eps}\frac{\eps}{\eta^3}=\frac{1}{\eta^2}\gg1$ when $\eta\ll1$. The so-called kinetic regime, corresponding to $\eta=1$ and an optical depth of order $O(1)$, could also be analyzed in the general setting of interest in this paper as in \cite{bal2024complex}. To simplify the presentation, we only consider the diffusive limit as $\eta\to0$. Note that the variable $x$ modeling the natural scale of the correlation length, is not rescaled. 

Under this scaling, the wavefield $u^\eps$ satisfies the SPDE
\begin{equation}\label{eqn:IS}
  \mathrm{d}u^\eps=\frac{i\eta}{2\eps\omega}\Delta_xu^\eps\mathrm{d}z-\frac{\omega^2R(0)}{8\eta^2}u^\eps\mathrm{d}z+\frac{i\omega}{2\eta}u^\eps\mathrm{d}B\,.  
\end{equation}
The aforementioned Gaussian structure can only be established for sufficiently broad incident beams of the form
\begin{equation}\label{eq:broad}
    u^\eps(0,x)=u^\eps_0(x)e^{i\sk\cdot x}=u_0(\eps x)e^{i\sk\cdot x}.
\end{equation}
Even broader beams $u_0(\eps^\beta x)$ for $\beta>1$ and their corresponding simplifications could be analyzed as in \cite{bal2024complex}. We focus on the richer case $\beta=1$. Here, $\sk$ models a lateral wavenumber of the incident wavefield, and we assume that $u_0\in\mathcal{S}(\mathbb{R}^d)$, the space of Schwartz functions on $\mathbb{R}^d$. 

In the kinetic regime with $\eta=1$, statistical moments up till four for broad beams have been shown to be consistent with those of a complex Gaussian distribution in~\cite{garnier2014scintillation, garnier2016fourth} for single frequency wavefields, and in~\cite{garnier2023speckle} to include shifts in frequency. An analysis of the Fourier transform of the wavefield itself, after compensating for a highly oscillatory phase is possible, and has been shown to be a complex Gaussian distribution after removing a deterministic component~\cite{bal2011asymptotics, gu2021gaussian}.

For fixed $(z,r,\omega)$, the Gaussian statistics of the macroscopic wavefield 
\begin{equation}\nonumber
    x\mapsto \upsilon^\eps(z,r,x)=u^\eps\Big(z,\frac{r}{\eps}+\eta x\Big)
\end{equation}
in the $\eps\to 0$ limit was established in~\cite{bal2024complex} for sources of the form $u^\eps_0(x)=\sum_{j=1}^Nf_j(\eps^\beta x)e^{ik_j\cdot x},\beta\ge 1$.  This paper focuses on the boundary conditions given by \eqref{eq:broad}. Generalizations of such results to partially coherent sources and to a wavelength-dependent paraxial model of wave propagation may be found in \cite{bal2025longPCB,bal2025long}. 
\medskip

Define $z=(z_1,\cdots,z_p,z'_1,\cdots,z'_q)$, $x=(x_1,\cdots,x_p,x'_1,\cdots,x'_q)$, $\omega=(\omega_1,\cdots,\omega_p,\omega'_1,\cdots,\omega'_q)$, as well as $\sk=(\sk_1,\cdots,\sk_p,\sk_1',\cdots,\sk'_q)$, for two integers $p,q\geq0$. The $p+q$th statistical moment of such fields is:
\begin{equation}\label{eq:mupq}
    \mu^\eps_{p,q}(z,x;\omega,\sk)=\mathbb{E}\prod_{j=1}^pu^\eps(z_j,x_j;\omega_j,\sk_j)\prod_{l=1}^qu^{\eps*}(z'_l,x'_l;\omega'_l,\sk'_l)\,,
\end{equation}
where we have explicitly indicated the dependence on the frequency and lateral wave number of the source in the last two arguments. The multiscale behavior is then encoded in the following scalings:
\begin{equation}\label{eqn:scalings}
    z_j=z_0+\eps\eta h_j,\quad \omega_j=\omega_0+\eps\eta\Omega_j,\quad \sk_j=\sk_0+\eps\kappa_j,\quad h_j\in\mathbb{R},\quad\Omega_j\in\mathbb{R},\quad\kappa_j\in\mathbb{R}^d .
\end{equation}
%\gb{Why are these variations nonnegative?}\an{they don't have to be}
For fixed $(z_0, r,\omega_0,\sk_0)$ we thus introduce the following rescaled wavefield:
\begin{equation}\nonumber
    \upsilon^\eps(h,x;\Omega,\kappa)=u^\eps\Big(z_0+\eps\eta h,\frac{r}{\eps}+\eta x;\omega_0+\eps\eta\Omega,\sk_0+\eps\kappa\Big)\,.
\end{equation}
More generally, we consider a vector of such macroscopic wavefields  evaluated at $(h_j,x_j;\Omega_j,\kappa_j)\in\mathbb{R}^{2d+2}$:
\begin{equation}
\Upsilon^\eps =   \{\upsilon^\eps_j(h_j,x_j;\Omega_j,\kappa_j)\}_{j=1}^N=\{u^\eps(z_0+\eps\eta h_j,\eps^{-1}r+\eta x_j;\omega_0+\eps\eta\Omega_j,\sk_0+\eps\kappa_j)\}_{j=1}^N\,.
\end{equation}
Note that the statistical moments of $\Upsilon^\eps$ are all of the form given in \eqref{eq:mupq}.

\medskip

If $\mathcal{Z}=\{\mathcal{Z}_1,\cdots,\mathcal{Z}_N\}$ is a mean zero complex Gaussian random vector, then all its higher moments are given by products of second moments, as
\begin{equation}\nonumber
    \mathbb{E}\prod_{j=1}^p\mathcal{Z}_{s_j}\prod_{l=1}^q\mathcal{Z}_{t_l}=\begin{cases}
        0,\quad p\neq q\\
        \sum_{\pi_p}\prod_{j=1}^p\mathbb{E}\mathcal{Z}_{s_j}\mathcal{Z}^\ast_{t_{\pi_p(j)}},\quad p=q\,.
    \end{cases}
\end{equation}
Here, $p$ and $q$ are non-negative integers, $\pi_p$ is a permutation of $p$ integers from $\{1,\cdots,p\}$ (without replacement), and $s_j$ and $t_l$ are integers drawn from $\{1,\cdots,N\}$.

\medskip

The main objective of this paper is to show that the process $\upsilon^\eps$ is mean-zero, complex Gaussian as $\eps\to0$ for fixed values of $(z_0,r,\omega_0,\sk_0)$ and that its characterizing correlation function solves an appropriate diffusion-type equation. The memory effect mentioned in the title refers to the speckle correlations as the parameters $(h,x,\Omega,\kappa)$ vary. Such memory effects play an important role in the analysis of wavebeam propagation through random media \cite{osnabrugge2017generalized,zhu2020chromato,garnier2023speckle}.

\medskip

The rest of the paper is structured as follows. The main results of the paper are presented in section \ref{sec:main}. As in \cite{bal2024complex}, the derivation is based on a convergence result for finite-dimensional moments given in Theorem \ref{thm:finite_dim_mom} and on a compactness argument presented in Theorem \ref{thm:tightness}. The main technical novelty compared to \cite{bal2024complex} stems from the fact that second-order moments, analyzed in detail in section \ref{sec:second}, no longer have closed form solutions. This generates a number of difficulties, already partially addressed in \cite{garnier2023speckle}, to establish moment convergence in section \ref{sec:higher} as well as tightness results in section \ref{sec:tight}.

\section{Main results}\label{sec:main}
\paragraph{Assumption on the medium.} We assume that the lateral covariance $R(x)=R(-x)\in \sL^1(\mathbb{R}^d)\cap\sL^\infty(\mathbb{R}^d)$ is a smooth, symmetric function with a strict maximum at $x=0$ (that it is a maximum stems from the fact that its Fourier transform $\hat R(k)\geq0$). This also implies that $\hat{R}(k)\in \sL^1(\mathbb{R}^d)\cap\sL^\infty(\mathbb{R}^d)$. Next, we assume that there exists a radially symmetric $\hat\sfR(k)\in\sL^1(\mathbb{R}^d)$ such that $\hat{R}(k)\le \hat\sfR(k)=\hat{\sfR}(|k|)$ and that for every $e\in\mathbb{S}^{d-1}$ and $\tau\in\mathbb{R}^d$, $s\mapsto R(\tau+se)$ is integrable in $\tau$. We finally assume that the Hessian $-\Sigma:=\nabla^2 R(0)$ is negative definite and $\hat R(k)$ decays sufficiently rapidly. More precisely, we assume that in lateral dimension $d\ge 3$, $\langle k\rangle^{d-2}\hat{R}(k)\in\sL^\infty(\mathbb{R}^d)$, where $\langle k\rangle:=\sqrt{1+|k|^2}$. %\gb{Exact assumptions? Should be summarized here.}\an{perhaps this should be a separate subsection?}
\paragraph{Convergence results.}
We now state our main convergence results as $\eps\to0$. 

\begin{theorem}[Convergence of finite dimensional distributions]\label{thm:finite_dim_mom}
The random vector $\Upsilon^\eps\Rightarrow\Upsilon$ in distribution as $\eps\to 0$ where $\Upsilon$ is a complex Gaussian random vector with entries $\{\upsilon_j\}_{j=1}^N$ satisfying
\begin{equation}\nonumber
    \mathbb{E}\upsilon_j=\mathbb{E}\upsilon_j\upsilon_l=0,\quad \mathbb{E}\upsilon_j\upsilon_l^\ast=m_{1,1}(z_0,r,h_{j,l},\tau_{j,l};\omega_0,\Omega_{j,l},\kappa_{j,l})\,,
\end{equation}
where $h_{j,l}=h_j-h_l$, $\tau_{j,l}=\tau_j-\tau_l$,  $\Omega_{j,l}=\Omega_l-\Omega_j$, $\kappa_{j,l}=\kappa_j-\kappa_l$, and $m_{1,1}$ is given by \eqref{eqn:m_11} below.
\end{theorem}

The Gaussian random field is fully characterized by its second-order moments. They are constructed as follows.  Let $M_{1,1}(z,r,\tau)$ be the solution to the following evolution equation:
\begin{equation}\label{eqn:M_11}
            \partial_zM_{1,1}=\frac{i\Omega}{2\omega_0^2}\Delta_\tau {M}_{1,1}+\frac{i}{\omega_0}\partial_r\cdot\partial_\tau M_{1,1}-\frac{\omega_0^2}{8}(\tau^\top\Sigma\tau) M_{1,1},\quad M_{1,1}(0,r,\tau)=|u_0(r)|^2e^{ir\cdot\kappa}.
\end{equation}
Then $m_{1,1}$ appearing in Theorem \ref{thm:finite_dim_mom}  is defined as
\begin{equation}\label{eqn:m_11}
        \begin{aligned}
         m_{1,1}(z_0,r,h,\tau;\omega_0,\Omega,\kappa)&= \Big(\frac{\omega_0}{2\pi i h}\Big)^{d/2}\int_{\mathbb{R}^d}M_{1,1}(z_0,r,\tau-\tau';\Omega,\kappa)e^{\frac{i\omega_0}{2h}|\tau'|^2}\mathrm{d}\tau'.
        \end{aligned}
\end{equation}
These moments fully characterize the limiting distribution $\Upsilon$. Convergence of the random field $(h,x,\Omega,\kappa)\mapsto \upsilon^\eps(h,x;\Omega,\kappa)$ is then a consequence of the following stochastic continuity result:
\begin{theorem}[Tightness and stochastic continuity]\label{thm:tightness}
   For $j=1,2$, let $(h_j,x_j,\Omega_j,\kappa_j)\in\mathbb{R}^{2d+2}$ such that $|h_1-h_2|, |x_1-x_2|, |\Omega_1-\Omega_2|, |\kappa_1-\kappa_2|< 1$.  We have
    \begin{equation}\label{eqn:ups_tight}
        \mathbb{E}|\upsilon^\eps(h_1,x_1;\Omega_1,\kappa_1)-\upsilon^\eps(h_2,x_2;\Omega_2,\kappa_2)|^{2n}\le C_\alpha(n,d)(|h_1-h_2|^n+|x_1-x_2|^{2n\alpha}+|\Omega_1-\Omega_2|^{2n\alpha}+|\kappa_1-\kappa_2|^{2n\alpha})\,,
    \end{equation}
    where 
    %$\alpha=1$ in the kinetic regime and is an arbitrary number in 
    $\alpha\in (0,1)$. Choosing $n$ large enough so that $ n\ge 2\alpha_-n+2d+2$ for arbitrary $\alpha_-\in(0,\frac12)$, we have that there exists a H\"{o}lder continuous version of $\upsilon^\eps$ on $C^{0,\alpha_-}(\mathbb{R}^{2d+2})$ and the process $\upsilon^\eps$ is tight on the same space.
\end{theorem}
The two theorems above then classically lead to the following result \cite{billingsley2017probability}:
\begin{theorem}[Convergence of processes]
    The processes whose finite dimensional distributions are shown to converge in Theorem~\ref{thm:finite_dim_mom} converge in distribution as probability measures on $C^{0,\alpha_-}(\mathbb{R}^{2d+2})$.
\end{theorem}

Theorem \ref{thm:finite_dim_mom} is proved in sections \ref{sec:second} and \ref{sec:higher} while Theorem \ref{thm:tightness} is proved in section \ref{sec:tight}. The computations of section \ref{sec:second} also provide explicit characterizations of the second-order moment $m_{1,1}$ in \eqref{eqn:m_11} that are useful in the understanding of spatial and frequential memory effects in speckle. %In particular we obtain the following.

\paragraph{Speckle memory effect.}

The moments $m_{1,1}$ in \eqref{eqn:m_11} do not seem to exhibit a closed-form expression in general. We consider several simplified scenarios of practical interest where explicit computations are feasible. 

Applications may be found in physical experiments of laser propagation through strongly scattering media, where a tilt in the source is observed to generate a tilt in the opposite direction in the generated speckle pattern. Similarly, shifts in the frequency lead to shifts in the speckle along the axis of propagation. These effects are collectively known as `memory effects' in the physical literature. We quantify two such memory effects, under the diffusive regime described previously and retrieve expressions of correlation functions first derived in \cite{osnabrugge2017generalized, zhu2020chromato} using a Feynman path integral formulation. Also, see~\cite{garnier2023speckle} for an analysis of decorrelation in frequency under the It\^o-Schr\"{o}dinger regime and their application to refocusing of signals in time reversal experiments. %\gb{References Knut et al.?}\an{I think their analysis mainly was to quantify loss of correlations w.r.t frequency shifts}

The first effect we wish to consider is the tilt memory effect \cite{osnabrugge2017generalized}. Suppose $\Omega=h=0$ and we tilt two sources symmetrically by $\Delta\kappa'/2$. We also tilt the wavefields at the receiver symmetrically by $-\Delta\kappa/2$. Such an effect is best modelled by considering the correlation function
\begin{equation}\nonumber
    \mathscr{C}_z(\tau,\Delta\kappa,\Delta\kappa')=\int_{\mathbb{R}^d}\mathbb{E}[u(z,r+\tau/2;\sk_1)u^\ast(z,r-\tau/2;\sk_2)]e^{-i\Delta\kappa\cdot r}\mathrm{d}r\,,
\end{equation}
where $\sk_1-\sk_2=\Delta\kappa$. We consider the diffusive scaling~\eqref{eqn:scint_scaling}, with the wavefield $u^\eps$ following the It\^o-Schr\"{o}dinger equation~\eqref{eqn:IS}. Let $m_{1,1}^\eps(z,r,\tau;\Delta\kappa')=\mathbb{E}[u^\eps(z,r/\eps+\eta\tau/2;\eps\Delta\kappa'/2)u^{\eps\ast}(z,r/\eps-\eta\tau/2;-\eps\Delta\kappa'/2)]$ and let $D_\sigma$ be the diffusion kernel
 \begin{equation}\label{eqn:D_sigma}
    D_{\sigma}(z,\tau,\xi)=\exp\big(-\frac{\omega_0^2\sigma^2 z}{8}\int_0^1\big|\tau+\frac{s\xi z}{\omega_0}\big|^2\mathrm{d}s\big)\,.
\end{equation}
Here we have assumed for simplicity, that $\Sigma=\sigma^2\mathbb{I}_d$, although more general noise models can be dealt with after minor modifications.  Denote by $\mathscr{C}_z^\eps$, the corresponding correlation. We then have the following.
 \begin{theorem}[Tilt memory effect]\label{thm:tilt_mem}
 In the diffusive regime, $\lim_{\eps\to 0}\mathscr{C}_z^\eps=\mathscr{C}_z$, where
\begin{equation}\nonumber
\begin{aligned}
  \mathscr{C}_z(\tau,\Delta\kappa,\Delta\kappa')&=D_{\sigma}(z,\tau,\Delta\kappa)\check{\Gamma}(\Delta\kappa-\Delta\kappa',0)\,,
\end{aligned}
\end{equation}
and $\check{\Gamma}(\kappa,0)=\int_{\mathbb{R}^d}|u_0(r)|^2e^{-ir\cdot\kappa}\mathrm{d}r$. Moreover, suppose $\check{\Gamma}(k,0)$ is maximal at $k=0$. Then the correlation $\mathscr{C}_z$ is maximized for the choice $\Delta\kappa'=\Delta\kappa=-\frac{3\omega_0\tau}{2z}$. 
\end{theorem}
In applications of microscopy and adaptive optics, it is desirable to increase spatial correlations as much as possible so as to maximize the scan range of scattered wavefields~\cite{mertz2015field, osnabrugge2017generalized}. For a choice of $(\Delta\kappa, \Delta\kappa')$ as in Theorem~\ref{thm:tilt_mem}, the optimal correlation $\mathscr{C}_{z,opt}=e^{-\frac{\sigma^2\omega_0^2z|\tau|^2}{32}}\check{\Gamma}(0,0)$. This leads to an improvement of a factor $2$ of the correlation width in $\tau$ compared to the  case of no tilt, where $\mathscr{C}_z(\tau,0,0)=e^{-\frac{\sigma^2\omega_0^2z|\tau|^2}{8}}\check{\Gamma}(0,0)$.  

Another interesting memory effect is related to lateral shifts in the source $\Delta x\mapsto u_0(\eps(x+\Delta x))$ and their relation to tilts \cite{osnabrugge2017generalized}. We do not discuss such a case here as they require spatially incoherent incident beams to have a meaningful order $O(1)$ effect, i.e. fields that vary at the scale $u_0(x/\eta)$ rather than the spatially very coherent fields of the form $u_0(\eps x)$ we considered so far. Spatially partially incoherent fields are in fact those of the form $u^\eps(z_0,x)$ for some $z_0>0$. Considering such fields as incident beams before applying spatial shifts and tilts would allow us to retrieve the general memory effects analyzed in \cite{osnabrugge2017generalized}. This effect will be analyzed in more detail elsewhere.

%We have not discussed memory effects due to lateral shifts in the source, because in order for shifts to have an $O(1)$ effect, the spatial correlations of the source should be at the same scale of the speckle. Sources of the form~\eqref{eq:broad} are too spatially coherent to achieve this, and have to be replaced by randomized sources as in~\cite{osnabrugge2017generalized}. Such effects can be emulated by propagating the source through a random medium up to a large enough distance, and then shifting it. Combinations of such shifts with tilts lead to better correlation properties at the receiver~\cite{osnabrugge2017generalized} and can be analyzed using the framework in this paper. Such an analysis will be considered elsewhere.\an{is this fine?}

%\gb{Improvement of $2$ visible where?}

 \medskip
 
We now discuss a chromato-axial memory effect in some detail. This effect was analyzed in \cite{zhu2020chromato}. Decorrelation effects in frequency for a fixed propagation range were also considered in \cite{garnier2023speckle}. 

%Here, we illustrate how such shifts influence correlations along the range.
%\gb{References Knut et al?}\an{is this fine?} 
We consider plane wave sources and noise $\Sigma=\sigma^2\mathbb{I}_d$, with generalizations to sources of the form~\eqref{eq:broad} and covariances $\Sigma\succ 0$ left to the reader. Let 
 \begin{equation}\nonumber
 a(z_0)=\frac{1}{2}\log[\cosh{\alpha_\Omega z_0}],\quad b(z_0)=\frac{\omega^2\sigma^2}{8\alpha_\Omega}\tanh{\alpha_\Omega z_0},\quad  \alpha_\Omega=e^{i\pi/4}\sqrt{\frac{\sigma^2\Omega}{4}}\,,   
 \end{equation}
 and let $\mathfrak{b}=b_R+i(b_I-\frac{\omega_0}{2h})$, where $b_R=Re(b), b_I=Im(b)$, with similar notation for $\mathfrak{b}$.% \gb{bizarre convention?} 
 We have the following.
\begin{theorem}[Chromato-axial memory effect]\label{thm:chrom_ax_mem}
When $\Sigma=\sigma^2\mathbb{I}_d$, and source $u^\eps(z=0)=1$, the two-point correlation~\eqref{eqn:m_11} is given by
\begin{equation}\nonumber
   m_{1,1}(h,\tau;\Omega)=\Big(\frac{\omega_0}{2ih\mathfrak{b}}\Big)^{d/2}e^{-a(z_0)}e^{-\frac{\omega_0^2\mathfrak{b}_R|\tau|^2}{4h^2|\mathfrak{b}|^2}}e^{\frac{i\omega_0|\tau|^2}{2h}(1+\frac{\omega_0\mathfrak{b}_I}{2h|\mathfrak{b}|^2})}\,. 
\end{equation}
  Moreover, for fixed $\Omega$, expanding $a$ and $b$ for small values of $\alpha_\Omega z_0$ gives that a choice of $h=-\frac{z_0\Omega}{3\omega_0}$ leads to larger correlations at $\tau=0$, a reflection of the chromato-axial memory effect in this regime.
\end{theorem}
The above result implies that choosing $h$ as above leads to a correlation $|m_{1,1}|=\big(1+\frac{b_I^2}{b_R^2}\big)e^{-a}$. When $h=0$, the corresponding value is $|m_{1,1}(h=0)|=e^{-a}$. So choosing $h$ optimally leads to an optimal improvement of a factor of $(1+\frac{b_I^2}{b_R^2})$ in the correlation factor and hence to a larger memory effect.
%\gb{Needs some words on why this is relevant. }\an{how about this?}

These results are derived in detail in section \ref{sec:second}.
%Present here diffusion equation and results on memory effects for shift/tild and for axio-chromatic effect. Details presented (presumably without pictures) in section \ref{sec:second}.
%
%%
\section{Analysis of second-order moments}
\label{sec:second}
The first-order moment is given for $(p,q)=(1,0)$ by $\mu^\eps_{1,0}(z,x;\omega,\sk)=\mathbb{E}u^\eps(z,x;\omega,\sk)$ and solves the equation
\begin{equation}\label{eqn:mu10}%\nonumber
    \partial_z\mu^\eps_{1,0}=\frac{i\eta}{2\eps\omega}\Delta_x\mu^\eps_{1,0}-\frac{\omega^2R(0)}{8\eta^2}\mu^\eps_{1,0},\quad \mu^\eps_{1,0}(0,x)=u_0^\eps(x)e^{i\sk\cdot x}\,.
\end{equation}
From, e.g., \cite{bal2024complex}, this equation can be solved explicitly using Fourier transforms to give
\begin{equation}\nonumber %\label{eqn:first_mom}
    m^\eps_{1,0}(z,r,x;\omega,\sk)=\mu^\eps_{1,0}(z,\frac r\eps+\eta x;\omega,\sk)=e^{-\frac{\omega^2R(0)z}{8\eta^2}}\int_{\mathbb{R}^d}\hat{u}_0(\xi)e^{-\frac{iz\eta}{2\eps\omega}|\sk+\eps\xi|^2}e^{i\eps^{-1}r\cdot(\sk+\eps \xi)}e^{i\eta x\cdot(\sk+\eps\xi)}\frac{\mathrm{d}\xi}{(2\pi)^d}\,,
\end{equation}
where $\hat{u}_0(\xi)=\int_{\mathbb{R}^d}u_0(x)e^{-i\xi\cdot x}\mathrm{d}x$ denotes the Fourier transform of $u_0$. This term would need to be analyzed in detail in the kinetic regime $\eta\approx1$. In the diffusive regime, where the coherence of the incident beam is exponentially suppressed, the mean field vanishes in the limit $\eta\to0$. This justifies the limit $\E\upsilon_j=0$ in Theorem \ref{thm:finite_dim_mom}.

\medskip

Define now the second-order moment for $(p,q)=(1,1)$ by
\begin{equation}\nonumber
    \mu^\eps_{1,1}(z_1,z_2,x,y;\omega_1,\omega_2,\sk_1,\sk_2)=\mathbb{E}u^\eps(z_1,x;\omega_1,\sk_1)u^{\eps\ast}(z_2,y;\omega_2,\sk_2)\,.
\end{equation}
We do not have a closed form equation available immediately for this function unless $z_1=z_2$. However, after an appropriate phase compensation, the above moment turns out to be the same as the phase compensated second-order moment evaluated at a fixed propagation distance $\min(z_1,z_2)$; see Lemma \ref{lemma:psi_11} below. %Such moments do satisfy a closed-form equation.

\subsection{Phase compensation}
Define the partial Fourier transform
$\hat{u}^\eps(z,\xi;\cdot)=\int_{\mathbb{R}^d}u^\eps(z,x;\cdot)e^{-i\xi\cdot x}\mathrm{d}x$.
 It solves
\begin{equation}\nonumber
    \mathrm{d}\hat{u}^\eps=-\frac{i\eta}{2\eps\omega}|\xi|^2\hat{u}^\eps\mathrm{d}z-\frac{\omega^2R(0)}{8\eta^2}\hat{u}^\eps\mathrm{d}z+\frac{i\omega}{2\eta}\int_{\mathbb{R}^d}\hat{u}^\eps(z,\xi-k)\frac{\mathrm{d}\hat{B}(z,k)}{(2\pi)^d},\quad \hat{u}^\eps(z=0,\xi;\omega,\sk)=\hat{u}^\eps_0(\xi-\sk)\,.
\end{equation}
Next we define the phase compensated field
\begin{equation}\label{eqn:psi_def}
    \psi^\eps(z,\xi;\omega,\sk)=\hat{u}^\eps(z,\xi;\omega,\sk)e^{\frac{i\eta z}{2\eps\omega}|\xi|^2}e^{\frac{\omega^2R(0)z}{8\eta^2}}\,.
\end{equation}
This solves the equation
\begin{equation}\label{eqn:phase_comp}
    \mathrm{d}\psi^\eps=\frac{i\omega}{2\eta}\int_{\mathbb{R}^d}\psi^\eps(z,\xi-k)e^{\frac{i\eta zg(\xi,k)}{2\eps\omega}}\frac{\mathrm{d}\hat{B}(z,k)}{(2\pi)^d},\quad \psi^\eps(0,\xi;\cdot)=\hat{u}^\eps(0,\xi;\cdot)\,,
\end{equation}
where the phase $g(\xi,k):=|\xi|^2-|\xi-k|^2=-|k|^2+2\xi\cdot k$. 
 For a product of two such phase compensated fields, let $\Psi_{1,1}^\eps(z_1,z_2,\xi,\zeta;\omega_1,\omega_2,\sk_1,\sk_2)=\mathbb{E}\psi^\eps(z_1,\xi;\omega_1,\sk_1)\psi^{\eps\ast}(z_2,\zeta;\omega_2,\sk_2)$. The first observation of this section is the following.
 \begin{lemma}\label{lemma:psi_11}
 The two-point correlation of the phase compensated field is given by
           \begin{equation}\nonumber
     \Psi_{1,1}^\eps(z_1,z_2,\xi,\zeta;\omega_1,\omega_2,\sk_1,\sk_2)=\Psi_{1,1}^\eps(z_1\wedge z_2,z_1\wedge z_2,\xi,\zeta;\omega_1,\omega_2,\sk_1,\sk_2)\,.
 \end{equation}
  \end{lemma}
\begin{proof}
      We have
 \begin{equation}\nonumber
 \begin{aligned}
     &\Psi_{1,1}^\eps(z_1,z_2,\xi,\zeta;\omega_1,\omega_2,\sk_1,\sk_2)= \Psi_{1,1}^\eps(0,0,\xi,\zeta;\omega_1,\omega_2,\sk_1,\sk_2)\\
     &+\frac{\omega_1\omega_2}{4\eta^2}\int_0^{z_1\wedge z_2}\int_{\mathbb{R}^d}\Psi_{1,1}^\eps(s,s,\xi,\zeta;\omega_1,\omega_2,\sk_1,\sk_2)e^{\frac{i\eta s}{2\eps}\big(\frac{g(\xi,k)}{\omega_1}-\frac{g(\zeta,k)}{\omega_2}\big)}\frac{\hat{R}(k)\mathrm{d}k\mathrm{d}s}{(2\pi)^d}.
 \end{aligned}
     \end{equation}
     However, note that the second-order moment of the phase compensated field at $z_1=z_2=z$ solves 
     \begin{equation}\nonumber
         \partial_z\Psi_{1,1}^\eps(z,z,\xi,\zeta;\omega_1,\omega_2,\sk_1,\sk_2)=\frac{\omega_1\omega_2}{4\eta^2}\int_{\mathbb{R}^d}\Psi_{1,1}^\eps(z,z,\xi,\zeta;\omega_1,\omega_2,\sk_1,\sk_2)e^{\frac{i\eta z}{2\eps}\big(\frac{g(\xi,k)}{\omega_1}-\frac{g(\zeta,k)}{\omega_2}\big)}\frac{\hat{R}(k)\mathrm{d}k}{(2\pi)^d}\,.
     \end{equation}
     This gives $\Psi_{1,1}^\eps(z_1,z_2,\xi,\zeta;\omega_1,\omega_2,\sk_1,\sk_2)=\Psi_{1,1}^\eps(z_1\wedge z_2,z_1\wedge z_2,\xi,\zeta;\omega_1,\omega_2,\sk_1,\sk_2)$. 
\end{proof}
Now, we do a phase recompensation and inverse Fourier transform back to the physical variables. In the physical domain, we define the two-point correlation of the macroscopic wavefield
\begin{equation}\label{eqn:m_11_def}
    m^\eps_{1,1}(z_1,z_2,r,\tau;\omega_1,\omega_2,\sk_1,\sk_2)=\mu^\eps_{1,1}(z_1,z_2,\eps^{-1}r+\eta\tau/2,\eps^{-1}r-\eta\tau/2;\omega_1,\omega_2,\sk_1,\sk_2)\,,
\end{equation}
where $z_j=z_0+\eps\eta h_j$, $\omega_j=\omega_0+\eps\eta\Omega_j, \sk_j=\sk_0+\eps\kappa_j$. Also, let $\Gamma(r,\tau):=u_0(r-\tau/2)u_0^\ast(r+\tau/2)$. Let $\|\cdot\|$ denote the total variation (TV) norm and $\|\cdot\|_p$ denote the standard $\sL^p$ norm. We have the following.
\begin{proposition}\label{prop:m_11_expn}
The two-point correlation $m^\eps_{1,1}$ is given by
     \begin{equation}\nonumber
        \begin{aligned}
    m_{1,1}^\eps(z_1,z_2,r,\tau;\cdot)&=\int\limits_{\mathbb{R}^{2d}}\hat{{M}}^\eps_{1,1}(z_1\wedge z_2,\zeta,\xi;\cdot)e^{-\frac{ih|\xi|^2}{2\omega_0}}e^{i(r\cdot\zeta+\tau\cdot \xi)}\frac{\mathrm{d}\xi\mathrm{d}\zeta}{(2\pi)^{2d}} + E_1^\eps\,,
        \end{aligned}
    \end{equation}  
where ${M}^\eps_{1,1}$ solves%% the second moment PDE
    \begin{equation}\label{eqn:M_11_eps}
        \begin{aligned}
               \partial_z{M}^\eps_{1,1}&=\frac{i\Omega}{2\omega_0^2}\Delta_\tau{M}^\eps_{1,1}+\frac{i}{\omega_0}\partial_r\cdot\partial_\tau {M}^\eps_{1,1}+\frac{\omega_0^2}{4\eta^2}[R(\eta\tau)-R(0)]{M}^\eps_{1,1}\\
                {M}^\eps_{1,1}(0,\cdot)&=\Gamma(r,\eps\eta\tau)e^{ir\cdot\kappa}e^{i\eta\tau\cdot\Bar{\sk}}\,,
        \end{aligned}
        \end{equation}
     $h=h_1-h_2$,  $\Omega=\Omega_2-\Omega_1$, $\kappa=\kappa_1-\kappa_2, \Bar{\sk}=\frac{\sk_1+\sk_2}{2}$ and $\sup_{0\le s\le z}\|E^\eps_{1}(s,\cdot)\|_\infty\le C\eps^\alpha$, $\alpha\in(0,1)$\,.
\end{proposition}
\begin{proof}
Note that the phase-compensated moment $\Psi^\eps_{1,1}$ is related to $\hat{\mu}^\eps_{1,1}$, the Fourier transform of $\mu^\eps_{1,1}$ as
\begin{equation}\nonumber
    \hat{\mu}^\eps_{1,1}(z_1,z_2,\xi,\zeta;\omega_1,\omega_2,\sk_1,\sk_2)=\Psi^\eps_{1,1}(z_1,z_2,\xi,\zeta;\omega_1,\omega_2,\sk_1,\sk_2)e^{-\frac{i\eta}{2\eps}\big(\frac{z_1}{\omega_1}|\xi|^2-\frac{z_2}{\omega_2}|\zeta|^2\big)}e^{-\frac{R(0)}{8\eta^2}(\omega_1^2z_1+\omega_2^2z_2)}\,.
\end{equation}
       Phase recompensating and inverse Fourier transforming $\Psi^\eps_{1,1}$ then gives
 \begin{equation}\nonumber
 \begin{aligned}
     \mu^\eps_{1,1}(z_1,z_2,x,y;\omega,\sk)&=e^{-\frac{R(0)}{8\eta^2}(\omega_1^2 (z_1-z_1\wedge z_2)+\omega_2^2(z_2-z_1\wedge z_2))}\\
     &\times\int_{\mathbb{R}^{2d}}\hat\mu_{1,1}^\eps(z_1\wedge z_2,z_1\wedge z_2,\xi,\zeta;\omega,\sk)e^{-\frac{i\eta}{2\eps}\big(\frac{(z_1-z_1\wedge z_2)|\xi|^2}{\omega_1}-\frac{(z_2-z_1\wedge z_2)|\zeta|^2}{\omega_2}\big)}e^{i(\xi\cdot x-\zeta\cdot y)}\frac{\mathrm{d}\xi\mathrm{d}\zeta}{(2\pi)^{2d}}\,.
 \end{aligned}
     \end{equation}
Since the macroscopic wavefield is evaluated at $(x,y)=\big(\frac{r}{\eps}+ \frac{\eta\tau}{2},\frac{r}{\eps}- \frac{\eta\tau}{2}\big)$, making the change of variables $(\xi,\zeta)=\big(\frac{\xi}{\eta}+\frac{\eps\zeta}{2},\frac{\xi}{\eta}-\frac{\eps\zeta}{2}\big)$ gives
     \begin{equation}\nonumber
        \begin{aligned}
          &  m_{1,1}^\eps(z_1,z_2,r,\tau;\cdot)=\mu^\eps_{1,1}\big(z_1,z_2,\frac{r}{\eps}+\frac{\eta\tau}{2},\frac{r}{\eps}-\frac{\eta\tau}{2};\cdot\big)=e^{-\frac{R(0)}{8\eta^2}(\omega_1^2 (z_1-z_1\wedge z_2)+\omega_2^2(z_2-z_1\wedge z_2))}\\
    &\int\limits_{\mathbb{R}^{2d}}\hat{\tilde{{M}}}^\eps_{1,1}(z_1\wedge z_2,\zeta,\xi;\cdot)e^{-\frac{i\eta}{2\eps}\big(\frac{(z_1-z_1\wedge z_2)|\xi/\eta-\eps\zeta/2|^2}{\omega_1}-\frac{(z_2-z_1\wedge z_2)|\xi/\eta+\eps\zeta/2|^2}{\omega_2}\big)}e^{i(r\cdot\zeta+\tau\cdot \xi)}\frac{\mathrm{d}\xi\mathrm{d}\zeta}{(2\pi)^{2d}}\,.
        \end{aligned}
    \end{equation}   
        Here we defined $\hat{\tilde{{M}}}^\eps_{1,1}(z,\zeta,\xi;\cdot)=(\frac{\eps}{\eta})^d\hat{\mu}^\eps_{1,1}\big(z,z,\frac{\xi}{\eta}+\frac{\eps\zeta}{2},\frac{\xi}{\eta}-\frac{\eps\zeta}{2};\cdot\big)$.
   This is the solution to
   \begin{equation}\label{eqn:M_11_tilde}
        \begin{aligned}
      \partial_z\tilde{{M}}^\eps_{1,1}&=\frac{i\eta^2\Omega}{8\omega_1\omega_2}(\eps^2\Delta_r+4\eta^{-2}\Delta_\tau)\tilde{{M}}^\eps_{1,1}+\frac{i\Bar{\omega}}{\omega_1\omega_2}\partial_r\cdot\partial_\tau \tilde{{M}}^\eps_{1,1}+\frac{\omega_1\omega_2}{4\eta^2}[R(\eta\tau)-R(0)]\tilde{{M}}^\eps_{1,1}-\frac{\eps^2\Omega^2R(0)}{8}\tilde{{M}}^\eps_{1,1}\\
      \tilde{{M}}^\eps_{1,1}(0,\cdot)&=\Gamma(r,\eps\eta\tau)e^{i\eps^{-1}r\cdot(\sk_1-\sk_2)}e^{i\eta\tau\cdot(\sk_1+\sk_2)/2}\,,
        \end{aligned}
    \end{equation}
 where $\bar{\omega}=\frac{\omega_1+\omega_2}{2}$. We observe that upon Fourier transforming $\tilde{{M}}^\eps_{1,1}$,
       \begin{equation}\nonumber
           \hat{\tilde{{M}}}^\eps_{1,1}(z,\zeta,\xi;\cdot)e^{\frac{i\Omega\eta^2\eps^2|\zeta|^2z}{8\omega_1\omega_2}}e^{\frac{\eps^2\Omega^2R(0)z}{8}}=\hat{\breve{M}}^\eps_{1,1}(z,\zeta,\xi;\cdot)\,,
       \end{equation}
       where $\breve{M}^\eps_{1,1}$ solves
       \begin{equation}\label{eqn:M_11_breve}
        \begin{aligned}
      \partial_z\breve{M}^\eps_{1,1}&=\frac{i\Omega}{2\omega_1\omega_2}\Delta_\tau\breve{M}^\eps_{1,1}+\frac{i\Bar{\omega}}{\omega_1\omega_2}\partial_r\cdot\partial_\tau \breve{M}^\eps_{1,1}+\frac{\omega_1\omega_2}{4\eta^2}[R(\eta\tau)-R(0)]\breve{M}^\eps_{1,1}\\
      \breve{M}^\eps_{1,1}(0,\cdot)&=\Gamma(r,\eps\eta\tau)e^{i\eps^{-1}r\cdot(\sk_1-\sk_2)}e^{i\eta\tau\cdot(\sk_1+\sk_2)/2}\,. 
        \end{aligned}
    \end{equation}
       This gives $\|\hat{\tilde{{M}}}^\eps_{1,1}-\hat{\breve{M}}^\eps_{1,1}\|\le c\langle\Omega\rangle\eps^2z\|(1+\eta^2|\zeta|^2\hat{\breve{M}}^\eps_{1,1})\|$. For any $p\in\mathbb{Z}^+$, the TV norm $\|(1+|\zeta|^p)\hat{\breve{M}}^\eps_{1,1}(z)\|$ is bounded by $\|(1+|\zeta|^p)\breve{{M}}^\eps_{1,1}(0)\|$ which is independent of $\eps$. Now without loss of generality let $z_1\ge z_2$. Then we have up to $\mathcal{O}(\eps^2)$ in the uniform sense, %\gb{in what sense?}\an{it is also in the T.V sense in Fourier},
\begin{equation}\nonumber
    \begin{aligned}
     m_{1,1}^\eps(z_1,z_2,r,\tau;\cdot)  &=e^{-\frac{\eps R(0)\omega_1^2h}{8\eta}}\int\limits_{\mathbb{R}^{2d}}\hat{\breve{M}}^\eps_{1,1}(z_2,\zeta,\xi;\cdot)e^{-\frac{i\eta^2h|\xi/\eta-\eps\zeta/2|^2}{2\omega_1}}e^{i(r\cdot\zeta+\tau\cdot \xi)}\frac{\mathrm{d}\xi\mathrm{d}\zeta}{(2\pi)^{2d}}\,. 
    \end{aligned}
\end{equation}
We approximate the exponentials $ e^{-\frac{\eps R(0)\omega_1^2h}{8\eta}}e^{-\frac{i\eta^2h|\xi/\eta-\eps\zeta/2|^2}{2\omega_1}}$ by $e^{\frac{-ih|\xi|^2}{2\omega_1}}$ as
\begin{equation}\nonumber
 |e^{-\frac{ih|\xi|^2}{2\omega_1}}( e^{-\frac{\eps R(0)\omega_1^2h}{8\eta}}e^{-\frac{ih}{2\omega_1}(\eps^2\eta^2|\zeta|^2/4-\eps\eta\xi\cdot\zeta)}-1)|\le C\big[\frac{\eps  |h|}{\eta}+|h|(\eps^2\eta^2|\zeta|^2+\eps\eta|\xi||\zeta|)\big]\,.
\end{equation}
    It can be shown from~\eqref{eqn:M_11_breve} that for any $p\in\mathbb{Z}^+$, % \gb{how? directly from \eqref{eqn:M_11_breve} presumably?}\an{yes} 
    \begin{equation}\nonumber
        \|\langle\xi\rangle^p\hat{\breve M}^\eps_{1,1}\|(z)\le \|\langle\xi\rangle^p\hat{\breve M}^\eps_{1,1}\|(0)e^{\frac{\omega_1\omega_2z}{4\eta^2}\int_{\mathbb{R}^d}(2^p\langle\eta k\rangle^p\hat{R}(k)-\hat{R}(k))\mathrm{d}k}\le \|\langle\xi\rangle^p\hat{\breve M}^\eps_{1,1}\|(0)e^{\frac{c(p)}{\eta^2}}\,.
    \end{equation}
    This bound increases exponentially in $\eta^{-2}$. However, multiplication by a factor of $\eps$ makes the contribution from this term smaller than $\eps^\alpha$ for any $0<\alpha<1$. $\omega_{1,2}$ can finally be replaced by $\omega_0$ up to an $\mathcal{O}(\eps^\alpha)$ term in a similar manner to give Eq.~\eqref{eqn:M_11_eps}.
    \end{proof}

The second-order moment PDE~\eqref{eqn:M_11_eps} does not have an analytical solution in general, but in the diffusive regime, it is asymptotically given by a diffusion equation as is shown next.
  
\subsection{Frequency correlations}
        \begin{lemma}\label{lemma:M_11}
        For $z>0$ and in the diffusive limit, the solution $M^\eps_{1,1}$ to the PDE described by~\eqref{eqn:M_11_eps} is given by $\lim_{\eps\to0}M^\eps_{1,1}(z,r,\tau)=M_{1,1}(z,r,\tau)$ where $M_{1,1}$ solves \eqref{eqn:M_11}.% the PDE
        %\begin{equation}\label{eqn:M_11}
       %     \partial_zM_{1,1}=\frac{i\Omega}{2\omega_0^2}\Delta_\tau {M}_{1,1}+\frac{i}{\omega_0}\partial_r\cdot\partial_\tau M_{1,1}-\frac{\omega_0^2}{8}(\tau^\top\Sigma\tau) M_{1,1},\quad M_{1,1}(0,r,\tau)=\Gamma(r,0)e^{ir\cdot\kappa}=|u_0(r)|^2e^{ir\cdot\kappa}\,.
      %  \end{equation}
       % Here, $\Sigma=-\nabla^2R(0)$.
        \begin{proof}
        Fourier transforming $(r,\tau)\to (\zeta,\xi)$ gives
\begin{equation}\nonumber
        \begin{aligned}
               \partial_z\hat{M}^\eps_{1,1}&=-iV(\xi)\hat{M}^\eps_{1,1}+\mathcal{L}^\eta\hat{M}^\eps_{1,1}\\
                \hat{M}^\eps_{1,1}(0,\cdot)&=(\eps\eta)^{-d}\hat{\Gamma}(\zeta-\kappa,(\eps\eta)^{-1}(\xi-\eta\Bar{\sk}))\,,
        \end{aligned}
        \end{equation}
        where (after suppressing the $\zeta$ dependence)
        \begin{equation}\nonumber
        \begin{aligned}
                     V(\xi)&=\frac{\Omega|\xi|^2}{2\omega_0^2}+\frac{(\xi\cdot\zeta)}{\omega_0},\quad [\mathcal{L}^\eta\psi](\xi)=\frac{\omega_0^2}{4\eta^2}\int_{\mathbb{R}^d}[\psi(\xi-\eta k)-\psi(\xi)]\frac{\hat{R}(k)\mathrm{d}k}{(2\pi)^d}\,.
        \end{aligned}
        \end{equation}
When $\Omega=0$, the above equation admits an explicit solution, which in the physical domain is given by \cite{bal2024complex,garnier2014scintillation}:
\begin{equation}\label{eqn:M_11_Omega=0}
\begin{aligned}
    M^\eps_{1,1}(z,r,\tau;\omega_0,\Omega=0,\cdot)=&\int_{\mathbb{R}^{2d}}u_0\big(r'+\eps\eta\tau/2+\eps\eta\xi z/2\omega_0\big)u^\ast_0\big(r'-\eps\eta\tau/2-\eps\eta\xi z/2\omega_0\big)e^{i\xi\cdot(r-r')}e^{ir'\cdot\kappa}\\
            &\times\exp\Big(\frac{\omega_0^2 z}{4\eta^2}\int_0^1Q\big(\eta\tau+\frac{\eta s\xi z}{\omega_0}\big)\mathrm{d}s\Big)\frac{\mathrm{d}\xi\mathrm{d}r'}{(2\pi)^d}, \quad Q(x)=R(x)-R(0)\,.  
\end{aligned}
  \end{equation}
This is no longer the case in general when $\Omega\neq 0$ and we need to use the structure of $\mathcal{L}^\eta$ as the generator of a jump process, which is also utilized in~\cite{garnier2023speckle}.

            Let $\rho_0(\xi)\in C_b(\mathbb{R}^d)$, the space of continuous and bounded functions on $\mathbb{R}^d$. Define $\rho^\eps$ to be the solution to the adjoint equation
            \begin{equation}\nonumber
                \partial_z\rho^\eps=-iV(\xi)\rho^\eps-\mathcal{L}^\eta\rho^\eps,\quad \rho^\eps(Z,\xi)=\rho_0(\xi)\,.
            \end{equation}
            Let $\langle\cdot,\cdot\rangle$ denote the standard inner product w.r.t $\xi$. By the construction of $\rho^\eps$ and the assumption that $\hat{R}$ is real valued and symmetric, $\partial_z\langle\hat{M}^\eps_{1,1}(z),\rho^\eps(z)\rangle=0$ which gives
            \begin{equation}\nonumber
                \langle\hat{M}^\eps_{1,1}(Z),\rho_0\rangle=\langle \hat{M}^\eps_{1,1}(0),\rho^\eps(0)\rangle\,.
            \end{equation}
        It can be shown that %\gb{? Is it direct or should it have more detail?}
            \begin{equation}\label{eqn:M_eps_11_conv}
                \lim_{\eps\to 0} \langle\hat{M}^\eps_{1,1}(Z),\rho_0\rangle=\lim_{\eps\to 0}\langle \hat{M}^\eps_{1,1}(0),\rho^\eps(0)\rangle=\langle \hat{\Gamma}(\zeta-\kappa,\xi),\lim_{\eps\to 0}\rho^\eps(0,0)\rangle=\langle \hat{M}_{1,1}(Z),\rho_0\rangle\,.
            \end{equation}
             The second equality involves approximating $\hat{M}^\eps_{1,1}(0,\cdot,\xi)$ by a Dirac distribution in $\xi$. From standard regularity theory, the solution $\rho^\eps$ stays bounded and continuous. Assuming that the initial condition $\hat{\Gamma}$ is smooth, from Lebesgue's dominated convergence theorem, %\an{Is this so straightforward?}
             \begin{equation}\nonumber
\lim_{\eps\to 0}\langle \hat{M}^\eps_{1,1}(0),\rho^\eps(0)\rangle=\lim_{\eps\to 0}\int_{\mathbb{R}^d}\hat{\Gamma}(\zeta-\kappa,\xi)\rho^{\eps\ast}(0,\eps\eta\xi+\eta\bar{\sk})\mathrm{d}\xi=\int_{\mathbb{R}^d}\hat{\Gamma}(\zeta-\kappa,\xi)\lim_{\eps\to 0}\rho^{\eps\ast}(0,0)\mathrm{d}\xi\,.
             \end{equation}
            As discussed in~\cite{garnier2023speckle}, the third equality in~\eqref{eqn:M_eps_11_conv} uses the fact that $\mathcal{L}^\eta$ is the infinitesimal generator of the random process $\chi^\eta(z)=\eta\chi(z/\eta^2)$ (where $\chi(z)$ is a compound Poisson process), and this converges to a Brownian motion $W$ with generator $\mathcal{L}=\frac{\omega_0^2}{8}\partial_\xi\cdot(\Sigma\partial_\xi)$ as $\eta\to 0$. This allows us to capture the limiting behavior of $\rho^\eps$ by writing an explicit expression using the Feynman-Kac representation:
            \begin{equation}\nonumber
                \rho^\eps(z,\xi)=\mathbb{E}[\rho_0(\chi^\eta(Z))e^{i\int_z^ZV(\chi^\eta(s))\mathrm{d}s}|\chi^\eta(z)=\xi]\xrightarrow{\eta\to 0}\mathbb{E}[\rho_0(W(Z))e^{i\int_z^ZV(W(s))\mathrm{d}s}|W(z)=\xi]=\rho(z,\xi)\,,
            \end{equation}
            where $\rho$ solves the backward diffusion equation
            \begin{equation}\nonumber
                 \partial_z\rho=-iV(\xi)\rho-\mathcal{L}\rho,\quad \rho(Z,\xi)=\rho_0(\xi)\,.
            \end{equation}
           Again, from the construction of $\hat{M}_{1,1}$ and $\rho$, we have
           \begin{equation}\nonumber
     \lim_{\eps\to 0}\langle \hat{\Gamma}(\zeta-\kappa,\xi),\rho^\eps(0,0)\rangle    = \langle \hat{\Gamma}(\zeta-\kappa,\xi),\rho(0,0)\rangle= \langle \hat{M}_{1,1}(Z),\rho_0\rangle \,.
           \end{equation}
            This shows that $\hat{M}^\eps_{1,1}(z,\zeta,\xi)$ converges as a bounded measure on $\mathbb{R}^{2d}$ to $\hat{M}_{1,1}(z,\zeta,\xi)$ as $\eps\to 0$ for $z>0$. Finally inverse Fourier transforming gives the statement in the Lemma.
        \end{proof}
    \end{lemma}
\subsection{Two-point correlation function in diffusive regime}
Summarizing the above derivations, we have the following corollary to Proposition~\ref{prop:m_11_expn} and Lemma~\ref{lemma:M_11}.
\begin{corollary}\label{coro:m_11_eps_limit}
    Under the diffusive regime and scaling $z_j=z_0+\eps\eta h_j$, $\omega_j=\omega_0+\eps\eta\Omega_j$ and $\sk_j=\eps\sk_0+\eps \kappa_j$, the two-point correlation $m^\eps_{1,1}$ described by~\eqref{eqn:m_11_def} is asymptotically given by $m_{1,1}(z_0,r,h,\tau;\omega_0,\Omega,\kappa)=\lim_{\eps\to 0} m^\eps_{1,1}(z_1,z_2,r,\tau;\omega_1,\omega_2,\sk_1,\sk_2)$ given in \eqref{eqn:m_11}. 
    %, where
    %\begin{equation}\label{eqn:m_11}
    %    \begin{aligned}
    %     m_{1,1}(z_0,r,h,\tau;\omega_0,\Omega,\kappa)&= \Big(\frac{\omega_0}{2\pi i h}\Big)^{d/2}\int_{\mathbb{R}^d}M_{1,1}(z_0,r,\tau-\tau';\Omega,\kappa)e^{\frac{i\omega_0}{2h}|\tau'|^2}\mathrm{d}\tau'\,,
     %   \end{aligned}
    %\end{equation}
   %Here, $h=h_1-h_2$, $\Omega=\Omega_2-\Omega_1$, $\kappa=\kappa_1-\kappa_2$ and $M_{1,1}$ solves the PDE~\eqref{eqn:M_11}.
\end{corollary}
We can now proceed to the proofs of the speckle memory effect in Theorems~\ref{thm:tilt_mem} and~\ref{thm:chrom_ax_mem}.
\paragraph{Lateral shifts and tilts and proof of Theorem~\ref{thm:tilt_mem}.}
\begin{proof}
  As $\Omega=h=0$, we have from Proposition~\ref{prop:m_11_expn} and~\eqref{eqn:M_11_Omega=0} that 
 \begin{equation}\nonumber
        \begin{aligned}
            m^\eps_{1,1}(z,z,r,\tau;0,\Delta\kappa)=&\int_{\mathbb{R}^{2d}}u_0\big(r'+\eps\eta\tau/2+\eps\eta\xi z/2\omega_0\big)u^\ast_0\big(r'-\eps\eta\tau/2-\eps\eta\xi z/2\omega_0\big)e^{i\xi\cdot(r-r')}e^{ir'\cdot\Delta\kappa'}\\
            &\times\exp\Big(\frac{\omega_0^2 z}{4\eta^2}\int_0^1Q\big(\eta\tau+\frac{\eta s\xi z}{\omega_0}\big)\mathrm{d}s\Big)\frac{\mathrm{d}\xi\mathrm{d}r'}{(2\pi)^d}\,.
        \end{aligned}
    \end{equation}
    Then we have that asymptotically, $\mathscr{C}^\eps_z(\tau,\Delta\kappa,\Delta\kappa')=\int_{\mathbb{R}^d}m^\eps_{1,1}(z,z,r,\tau;0,\Delta\kappa)e^{-ir\cdot\Delta\kappa}\mathrm{d}r$ is given by $\lim_{\eps\to 0}\mathscr{C}^\eps_z=\mathscr{C}_z$, where
\begin{equation}\label{eqn:C}
    \begin{aligned}
        \mathscr{C}_z(\tau,\Delta\kappa,\Delta\kappa')=D_{\sigma}(z,\tau,\Delta\kappa)\check{\Gamma}(\Delta\kappa-\Delta\kappa',0),\quad \check{\Gamma}(\kappa,0)=\int_{\mathbb{R}^d}|u_0(r)|^2e^{-ir\cdot\kappa}\mathrm{d}r\,,
    \end{aligned}
\end{equation}
and $D_\sigma$ is the diffusion kernel given by~\eqref{eqn:D_sigma}. Completing squares gives
\begin{equation}\nonumber
    \begin{aligned}
        D_{\sigma}(z,\tau,\xi)=\exp\big(-\frac{\omega_0^2\sigma^2 z}{8}\Big[|\tau|^2+\frac{z^2|\Delta\kappa|^2}{3\omega_0^2}+\frac{z\Delta\kappa\cdot\tau}{\omega_0}\Big]\big)=\exp\Big(-\frac{\sigma^2z^3}{24}\Big|\Delta\kappa+\frac{3\omega_0\tau}{2z}\Big|^2\Big)e^{-\frac{\sigma^2\omega_0^2z|\tau|^2}{32}}\,.
    \end{aligned}
\end{equation}
If $\check{\Gamma}(k,0)$ is maximal at $k=0$, the choice $\Delta\kappa'=\Delta\kappa$ leads to higher values of correlations. Moreover, the choice $\Delta\kappa=\Delta\kappa'=-\frac{3\omega_0\tau}{2z}$ maximizes $\mathscr{C}_z$, with the optimal value correlation being $\mathscr{C}_{z,opt}=e^{-\frac{\sigma^2\omega_0^2z|\tau|^2}{32}}\check{\Gamma}(0,0)$. This also improves the correlation width by a factor of 2, as opposed to the case when there are no tilts, as $\mathscr{C}_z(\tau,0,0)=e^{-\frac{\sigma^2\omega_0^2z|\tau|^2}{8}}\check{\Gamma}(0,0)$.
\end{proof}
\paragraph{Frequency and axial shifts and proof of Theorem~\ref{thm:chrom_ax_mem}.}
\begin{proof}
    Note that $M_{1,1}$ appearing in the expression of the asymptotic two-point correlation $m_{1,1}$~\eqref{eqn:m_11} can be re-written as:
\begin{equation}\nonumber
    M_{1,1}(z,r,\tau;\omega,\Omega,\kappa)=\int_{\mathbb{R}^d}\breve{M}_{1,1}(z,\zeta,\tau-\frac{\zeta z}{\omega_0};\omega,\Omega,\kappa)e^{i\zeta\cdot r}\frac{\mathrm{d}\zeta}{(2\pi)^d}\,,
\end{equation}
where $\breve{M}_{1,1}$ is given by the solution to
\begin{equation} \nonumber
    \begin{aligned}
        \partial_z\breve{M}_{1,1}=\frac{i\Omega}{2\omega_0^2}\Delta_\tau\breve{M}_{1,1}-\frac{z^2\sigma^2|\zeta|^2}{8}\breve{M}_{1,1}-\frac{z\omega_0\sigma^2}{4}(\zeta\cdot\tau)\breve{M}_{1,1}-\frac{\omega_0^2\sigma^2}{8}|\tau|^2\breve{M}_{1,1},\quad \breve{M}_{1,1}(0,\zeta,\tau)=\check{\Gamma}(\zeta-\kappa,0)\,,
    \end{aligned}
\end{equation}
with $\check{\Gamma}(\zeta,0)=\int_{\mathbb{R}^d}\hat{\Gamma}(\zeta,\xi)\frac{\mathrm{d}\xi}{(2\pi)^d}$.
This can be solved explicitly to get~\cite{zhu2020chromato, garnier2023speckle}
\begin{equation}\label{eqn:M_11_explicit}
    \breve{M}_{1,1}(z,\zeta,\tau;\omega,\Omega,\kappa)=\check{\Gamma}(\zeta-\kappa,0)e^{-[a(z)+b(z)|\tau|^2+c(z)\tau\cdot\zeta+d(z)|\zeta|^2]}
\end{equation}
where
\begin{equation}\label{eqn:abcd}
    \begin{aligned}
        a'(z)&=\frac{i\Omega b}{\omega_0^2},\quad b'(z)=-\frac{2i\Omega b^2}{\omega_0^2}+\frac{\omega_0^2\sigma^2}{8}\\
        c'(z)&=-\frac{2i\Omega bc}{\omega_0^2}+\frac{\omega_0\sigma^2 z}{4},\quad d'(z)=-\frac{i\Omega c^2}{2\omega_0^2}+\frac{\sigma^2z^2}{8}\,,
    \end{aligned}
\end{equation}
with initial condition $a(0)=b(0)=c(0)=d(0)=0$. The scalars $a,b,c,d$ are complex valued with non-negative real parts. For plane wave sources with $u_0=1$, we have $\check{\Gamma}(\zeta,0)=(2\pi)^d\delta_0(\zeta)$ and $\kappa=0$ so that $M_{1,1}(z,r,\tau;\omega,\Omega)=e^{-a(z)-b(z)|\tau|^2}$ and the two-point correlation
\begin{equation} \nonumber
\begin{aligned}
    m_{1,1}(h,\tau;\Omega)&=\Big(\frac{\omega_0}{2\pi i h}\Big)^{d/2}e^{-a(z_0)}\int_{\mathbb{R}^d}e^{-b(z_0)|\tau'^2|}e^{\frac{i\omega_0}{2h}|\tau-\tau'|^2}\mathrm{d}\tau'\\
    &=\frac{e^{-a(z_0)}}{(2ih/\omega)^{d/2}[b_R+i(b_I-\omega/2h)]^{d/2}}e^{\frac{i\omega_0}{2h}|\tau|^2}e^{-\frac{i\omega_0^2|\tau|^2(-b_I+\omega_0/2h)}{4h^2[b_R^2+(b_I-\omega_0/2h)^2]}}e^{-\frac{\omega_0^2b_R|\tau|^2}{4h^2[b_R^2+(b_I-\omega_0/2h)^2]}}\,,
\end{aligned}
\end{equation}
with $b=b_R+ib_I$ denoting the real and imaginary parts of $b$. The magnitude of this is
\begin{equation}\nonumber
    \begin{aligned}
        |m_{1,1}(h,\tau;\Omega)|=\frac{|e^{-a(z_0)}|}{(2h/\omega)^{d/2}[b_R^2+(b_I-\omega/2h)^2]^{d/4}}e^{-\frac{b_R\omega^2|\tau|^2}{4h^2[b_R^2+(b_I-\omega/2h)^2]}}\,.
    \end{aligned}
\end{equation}
The ODEs for $a$ and $b$ can be solved to get 
\begin{equation}\nonumber
    a(z)=\frac{1}{2}\log[\cosh{\alpha_\Omega z}],\quad b(z)=\frac{\omega^2\sigma^2}{8\alpha_\Omega}\tanh{\alpha_\Omega z},\quad\alpha_\Omega=e^{i\pi/4}\sqrt{\frac{\sigma^2\Omega}{4}}\,.
\end{equation}
At $\tau=0$, $|m_{1,1}|$ can be re-written as
\begin{equation}\nonumber
    |m_{1,1}(h,\tau;\Omega)|=\frac{|e^{-a(z_0)}|}{\Big(\frac{4}{\omega^2}(b_R^2+b_I^2)\big|h-\frac{b_I\omega}{2(b_R^2+b_I^2)}\big|^2+\frac{b_R^2}{(b_R^2+b_I^2)}\Big)^{d/4}}\,.
\end{equation}
For fixed $\Omega$, the optimal $h$ is then given by
\begin{equation}\nonumber
    h_{opt}=\frac{b_I\omega}{2(b_R^2+b_I^2)}\,.
\end{equation}
For small $|\alpha_\Omega z|=\frac{\sigma z\sqrt{\Omega}}{2}$, $b$ can be approximated as
\begin{equation}\nonumber
b\approx \frac{\omega^2\sigma^2 z}{8}\Big(1-\frac{i\sigma^2z^2\Omega}{12}\Big)\,.
\end{equation}
This gives $h_{opt}\approx-\frac{z\Omega}{3\omega[1+(\sigma^2z^2\Omega^2/12)^2]}\approx-\frac{z\Omega}{3\omega}$ for small $\Omega$.
\end{proof}

\section{Analysis and convergence of higher moments}
\label{sec:higher}
Here we analyze higher order statistical moments of the type
\begin{equation}\nonumber
    \mu^\eps_{p,q}(z,x;\omega,\sk)=\mathbb{E}\prod_{j=1}^pu^\eps(z_j,x_j;\omega_j,\sk_j)\prod_{l=1}^qu^{\eps\ast}(z'_l,x'_l;\omega'_l,\sk'_l)\,.
\end{equation}
When $z_j=z'_l=z$, these satisfy the following closed form PDEs from a generalization in~\cite{garnier2016fourth}:
\begin{equation}\label{eqn:mu_pq_PDE_fixed_z}
\begin{aligned}
    \partial_z\mu^\eps_{p,q}&=\frac{i\eta}{2\eps}(\sum_{j=1}^p\frac{1}{\omega_j}\Delta_{x_j}-\sum_{l=1}^q\frac{1}{\omega'_l}\Delta_{y_l})\mu^\eps_{p,q}+\mathcal{U}_{p,q}\mu^\eps_{p,q}\\
    \mu^\eps_{p,q}(0,x;\omega,\kappa)&=\prod_{j=1}^pu_0(\eps x_j)\prod_{j=1}^qu_0^\ast(\eps y_l)e^{i(\sum_{j=1}^p\sk_j\cdot x_j-\sum_{l=1}^q\sk'_l\cdot y_l)}\,,
\end{aligned}
    \end{equation}
where the potential $\mathcal{U}_{p,q}$ is given by
\begin{equation}\nonumber
    \begin{aligned}
        \mathcal{U}_{p,q}&=\sum_{j=1}^p\sum_{l=1}^q\frac{\omega_j\omega'_l}{4\eta^2}R(x_j-y_l)-\hspace{-0.4cm}\sum_{1\le j<j'\le p}\frac{\omega_j\omega_{j'}}{4\eta^2}R(x_j-x_{j'})-\hspace{-0.4cm}\sum_{1\le l<l'\le q}\frac{\omega'_l\omega'_{l'}}{4\eta^2}R(y_l-y_{l'})-(\sum_{j=1}^p\omega_j^2+\sum_{l=1}^q{\omega'_l}^2)\frac{R(0)}{8\eta^2}\,.
    \end{aligned}
\end{equation}
These moments do not satisfy closed form PDEs in the general case when $z_j=z'_l=z$. However, as will be shown in the subsequent computations, appropriately phase compensated moments at different $z$s are close to the phase compensated moments at the same $z$, provided the perturbation in $z$ is small. 

\subsection{Phase compensation}
For $z\ge z_0$, writing the phase compensated field~\eqref{eqn:phase_comp} as a Duhamel expansion starting at $z_0$ gives
 \begin{equation}\label{eqn:psi_exp}
     \psi^\eps(z,\xi;\omega,\sk)=\psi^\eps(z_{0},\xi;\omega,\sk)+\sum_{n\ge 1}\psi_n^\eps(z,\xi;\omega,\sk)\,,
 \end{equation}
 where
 \begin{equation}\nonumber
 \begin{aligned}
     \psi^\eps_n(z,\xi;\omega,\sk)=\Big(\frac{\omega}{2\eta}\Big)^ne^{\frac{in\pi}{2}}\int_{[z_{0},z]_<^n}\hspace{-0.1cm}\int_{\mathbb{R}^{nd}}&e^{\frac{i\eta}{2\eps\omega}[s_1g(\xi,k_1)+s_2g(\xi-k_1,k_2)+\cdots+s_ng(\xi-k_1-\cdots-k_{n-1},k_n)]}\\
     &\psi^\eps(z_0,\xi-k_1-\cdots-k_n;\omega,\sk)\prod_{j=1}^n\frac{\mathrm{d}\hat{B}(s_j,k_j)}{(2\pi)^d}.
 \end{aligned}
 \end{equation}
Let $v=(\xi_1,\cdots,\xi_p,\zeta_1,\cdots,\zeta_q)\in\mathbb{R}^{(p+q)d}$ be the vector of dual variables and $\Psi^\eps_{p,q}$ denote the expectation of $p+q$ such phase compensated fields, i.e,
\begin{equation}\label{eqn:Psi_pq_def}
    \Psi^\eps_{p,q}(z,v;\omega,\sk)=\mathbb{E}\prod_{j=1}^p\psi^\eps(z_j,\xi_j;\omega_j,\sk_j)\prod_{l=1}^q\psi^{\eps\ast}(z'_l,\zeta_l;\omega'_l,\sk'_l)\,.
\end{equation}
The main result of this section is the following.
 \begin{lemma}\label{lem:phase_comp_pq}
     Let $z_{\max}=\max\{z_j,z'_l\}$, $z_{\min}=\min\{z_j,z'_l\}$, $\omega_{\max}=\max\{\omega_j,\omega'_l\}$. %\gb{Have these objects been defined?} 
     The phase compensated moments can be decomposed as
     \begin{equation}\nonumber
         \Psi^\eps_{p,q}(z_1,\cdots,z_p,z'_1,\cdots,z'_q,v;\omega,\sk)=\Psi^\eps_{p,q}(z_{0},\cdots,z_{0},v;\omega,\sk)+E^\eps_2\,,
     \end{equation}
     where the error 
     \begin{equation}\nonumber
        \sup_{z_j,z'_l\in[z_{\min},z_{\max}]} \|E^\eps_2(z,v;\cdot)\|\le 2\|\Psi^\eps_{p,q}(z_{\min},\cdots,z_{\min},v;\omega,\sk)\|(e^{\frac{\mathfrak{C}(z_{\max}-z_{\min})}{\eta^2}}-1)\,,
     \end{equation}
     and $\mathfrak{C}=\frac{1}{8}\omega_{\max}^2R(0)(p+q)^2$. In particular, under the scaling~\eqref{eqn:scalings} we have $z_{\max}-z_{0}\le c\eps\eta$ so that
     \begin{equation}\nonumber
         \sup_{z_j,z'_l\in[z_{\min},z_{\max}]} \|E^\eps_2(z,v;\cdot)\|\le C\frac{\eps}{\eta}\|\Psi^\eps_{p,q}(z_{\min},\cdots,z_{\min},v;\omega,\sk)\|\,.
     \end{equation}
     \begin{proof}
         Using the Duhamel expansion starting at $z=z_{\min}$,
         \begin{equation}\nonumber
             \begin{aligned}
          &       \prod_{j=1}^p\psi^\eps(z_j,\xi_j;\cdot)\prod_{l=1}^q\psi^{\eps\ast}(z'_l,\zeta_l;\cdot)= \prod_{j=1}^p\psi^\eps(z_{\min},\xi_j;\cdot)\prod_{l=1}^q\psi^{\eps\ast}(z_{\min},\zeta_l;\cdot)\\
          &+\sum_{n_1+\cdots+n_{p+q}=n>0}\frac{\prod_{j=1}^p\omega_j^{n_j}\prod_{l=1}^q{\omega'_l}^{n_{p+l}}}{(2\eta)^n}\int_{[z_{\min},z_1]^{n_1}_<}\cdots\int_{[z_{\min},z'_{q}]_<^{n_{p+q}}}\int_{\mathbb{R}^{nd}}e^{\frac{i\eta G_n(\vec{s},\vec{k},v;\omega)}{2\eps}}\\
          &\prod_{j=1}^p\psi^\eps(z_{\min},\xi_j-k_{j,1}-\cdots -k_{j,n_j};\sk_j)\prod_{l=1}^{q}\psi^{\eps\ast}(z_{\min},\zeta_l+k'_{l,1}+\cdots+k'_{l,n_{p+l}};\sk'_l)\\
          &\prod_{j=1}^p\prod_{j'=1}^{n_j}\frac{\mathrm{d}\hat{B}(s_{j,j'},k_{j,j'})}{(2\pi)^d}\prod_{l=1}^q\prod_{l'=1}^{n_{p+l}}\frac{\mathrm{d}\hat{B}(s'_{l,l'},k'_{l,l'})}{(2\pi)^d}\,,
             \end{aligned}
         \end{equation}
         where $(\vec{s},\vec{k})=(s_{1,1},\cdots,s_{p,n_p},s'_{1,1},\cdots,s'_{q,n_{p+q}},k_{1,1},\cdots,k_{p,n_p},k'_{1,1},\cdots,k'_{q,n_{p+q}})$ and $G_n$ is the real valued phase
         \begin{equation}\nonumber
             \begin{aligned}
                 G_n(\vec{s},\vec{k},v;\omega)&=\sum_{j=1}^p\sum_{j'=1}^{n_j}\frac{s_{j,j'}}{\omega_j}g(\xi_j-k_{j,1}-\cdots-k_{j,{j'-1}},k_{j,j'})\\
                 &-\sum_{l=1}^q\sum_{l'=1}^{n_{p+l}}\frac{s'_{l,l'}}{\omega'_l}g(\zeta_l+k'_{l,1}+\cdots+k'_{l,{l'-1}},-k'_{l,l'})+\frac{\eps\pi}{\eta}(n_1+\cdots+n_p-n_{p+1}-\cdots-n_{p+q})\,.
             \end{aligned}
         \end{equation}
         Taking expectations, %%% \gb{A $+$ missing, here and a few times below?}\an{it is a product}
         \begin{equation}\nonumber
             \begin{aligned}
              & \|  \Psi^\eps_{p,q}(z,v;\cdot)-\Psi^\eps_{p,q}(z_{\min},\cdots,z_{\min},v;\cdot)\|\le \|\Psi^\eps_{p,q}(z_{\min},\cdots,z_{\min},v;\cdot)\|\\
               & \sum_{n_1+\cdots+n_{p+q}=2n>0}\Big(\frac{\omega_{\max}}{2\eta}\Big)^{2n}\int_{[z_{\min},z_{\max}]^{n_1}_<}\cdots\int_{[z_{\min},z_{\max}]^{n_{p+q}}_<}\int_{\mathbb{R}^{2nd}}|\mathbb{E}\prod_{j=1}^{2n}\frac{\mathrm{d}\hat{B}(s_j,k_j)}{(2\pi)^d}|\,.
             \end{aligned}
         \end{equation}
         The integrand above is symmetric with respect to permutations, so we have
             \begin{equation}\nonumber
             \begin{aligned}
              & \|  \Psi^\eps_{p,q}(z,v;\cdot)-\Psi^\eps_{p,q}(z_{\min},\cdots,z_{\min},v;\cdot)\|\le \|\Psi^\eps_{p,q}(z_{\min},\cdots,z_{\min},v;\cdot)\|\\
               & \sum_{n_1+\cdots+n_{p+q}=2n>0}\Big(\frac{\omega_{\max}}{2\eta}\Big)^{2n}\frac{1}{n_1!\cdots n_{p+q}!}\int_{[z_{\min},z_{\max}]^{2n}}\int_{\mathbb{R}^{2nd}}|\mathbb{E}\prod_{j=1}^{2n}\frac{\mathrm{d}\hat{B}(s_j,k_j)}{(2\pi)^d}|\\
               &=\|\Psi^\eps_{p,q}(z_{\min},\cdots,z_{\min},v;\cdot)\|\sum_{n_1+\cdots+n_{p+q}=2n>0}|P^{2n}|\Big(\frac{\omega_{\max}}{2\eta}\Big)^{2n}\frac{(z_{\max}-z_{\min})^nR(0)^n}{n_1!\cdots n_{p+q}!}\,,
(e^{\frac{\mathfrak{C}(z_{\max}-z_{\min})}{\eta^2}}-1)\,.             \end{aligned}
         \end{equation}     
        where $|P^{2n}|=\frac{(2n)!}{n!2^n}$ is the number of possible pairings of $2n$ Gaussian random variables. As a result of the combinatorial observation in Lemma~\ref{lemma:fact_sum} stated next, we have
        \begin{equation}\nonumber
            \| \Psi^\eps_{p,q}(z_1,\cdots,z_p,z'_1,\cdots,z'_q,v;\omega,\sk)-\Psi^\eps_{p,q}(z_{\min},\cdots,z_{\min},v;\omega,\sk)\|\le\|\Psi^\eps_{p,q}(z_{\min},\cdots,z_{\min},v;\omega,\sk)\| (e^{\frac{\mathfrak{C}(z_{\max}-z_{\min})}{\eta^2}}-1)\,.
        \end{equation}
        The proof is completed after approximating $\Psi^\eps_{p,q}(z_{\min},\cdots,z_{\min},v;\omega,\sk)$ by $\Psi^\eps_{p,q}(z_{0},\cdots,z_{0},v;\omega,\sk)$ in the total variation norm.
     \end{proof}
 \end{lemma}
 We verify a combinatorial bound of the following type used in the proof of Lemma~\ref{lem:phase_comp_pq} above. 
 \begin{lemma}\label{lemma:fact_sum}
     For a constant $c\ge 0$ and $p$ a positive integer,
     \begin{equation}\nonumber
         \sum_{n_1+\cdots+n_p=2n\ge 0}\frac{c^n(2n)!}{n!n_1!\cdots n_p!}=e^{p^2c}\,.
     \end{equation}
     \begin{proof}
     The sum above can be written as
     \begin{equation}\nonumber
     \begin{aligned}
     &   \sum_{n\ge 0}\frac{c^n}{n!}\sum_{n_1+\cdots+n_{p-1}=0}^{2n}\frac{(2n)!}{n_1!\cdots n_{p-1}!(2n-n_{p-1}-\cdots-n_1)!}\\
     &=\sum_{n\ge 0}\frac{c^n}{n!}\sum_{n_1+\cdots+n_{p-2}=0}^{2n}\frac{(2n)!}{n_1!\cdots n_{p-2}!}\sum_{n_{p-1}=0}^{2n-n_{p-2}-\cdots-n_1}\frac{1}{n_{p-1}!(2n-n_{p-1}-\cdots-n_1)!}\,.
     \end{aligned}
         \end{equation}
        Using $\sum_{m=0}^N\frac{b^{N-m}}{m!(N-m)!}=\frac{(1+b)^N}{N!}$, the sum above is
        \begin{equation}\nonumber
        \begin{aligned}
       \sum_{n\ge 0}\frac{c^n}{n!}\sum_{n_1+\cdots+n_{p-2}=0}^{2n}\frac{(2n)!2^{2n-n_{p-2}-\cdots-n_1}}{n_1!\cdots n_{p-2}!(2n-n_{p-2}-\cdots-n_1)!}\,.
        \end{aligned}
            \end{equation}
            Repeating this upto $n_2$ gives
            \begin{equation}\nonumber
                \sum_{n\ge 0}\frac{c^n}{n!}\sum_{n_1=0}^{2n}\frac{(2n)!(p-1)^{2n-n_1}}{n_1!(2n-n_1)!}=\sum_{n\ge 0}\frac{c^np^{2n}}{n!}=e^{p^2c}\,.
            \end{equation}
     \end{proof}
 \end{lemma}
 A consequence of Lemma~\ref{lem:phase_comp_pq} is that the statistical moments at different $z$s can be analyzed using the statistical moments at the same $z$ after phase compensating in the Fourier domain and phase recompensating and inverse Fourier transforming to the physical domain. This will require an adaptation of~\cite[Theorem 4.1]{bal2024complex} to the multi-frequncy case.
\subsection{Multi-frequency moments at fixed distance}
Denote by $\phi^\eps_{p,q}$ the $p+q$th moment at a fixed $z$ for different frequencies, i.e,
\begin{equation}\nonumber
    \phi^\eps_{p,q}(z,x;\omega,\sk)=\mathbb{E}\prod_{j=1}^pu^\eps(z,x_j;\omega_j,\sk_j)\prod_{l=1}^qu^{\eps\ast}(z,y_l;\omega'_l,\sk'_l)\,.
\end{equation}
This solves the PDE~\eqref{eqn:mu_pq_PDE_fixed_z}. Fourier transforming, we have
\begin{equation}\nonumber
    \begin{aligned}
        &\partial_z\hat{\phi}^\eps_{p,q}=-\frac{i\eta}{2\eps}\Big(\sum_{j=1}^p\frac{|\xi_j|^2}{\omega_j}-\sum_{l=1}^q\frac{|\zeta|^2}{\omega'_l}\Big)\hat{\phi}^\eps_{p,q}+\frac{1}{4\eta^2}\int_{\mathbb{R}^d}\Big[\sum_{j=1}^p\sum_{l=1}^q\omega_j\omega'_l\hat{\phi}^\eps_{p,q}(\xi_j-k,\zeta_l-k)\\
        &-\sum_{1\le j<j'\le p}\omega_j\omega_{j'}\hat{\phi}^\eps_{p,q}(\xi_j-k,\xi_{j'}+k)-\sum_{1\le l<l'\le q}\omega'_l\omega'_{l'}\hat{\phi}^\eps_{p,q}(\zeta_l-k,\zeta_{l'}+k)-\frac{1}{2}(\sum_{j=1}^p\omega_j^2+\sum_{l=1}^q{\omega'}^2_{l})\hat{\phi}^\eps_{p,q}\Big]\frac{\hat{R}(k)\mathrm{d}k}{(2\pi)^d}\\
        &\hat{\phi}^\eps(0)=\eps^{-(p+q)\beta d}\prod_{j=1}^p\hat{u}_0\Big(\frac{\xi_j-\sk_j}{\eps}\Big)\prod_{l=1}^q\hat{u}^\ast_0\Big(\frac{\zeta_l-\sk'_l}{\eps}\Big)\,.
    \end{aligned}
\end{equation}
Define the phase compensated moment
\begin{equation}\nonumber
    \varphi^\eps_{p,q}(z,v;\omega,\sk)=\hat{\phi}^\eps_{p,q}(z,v;\omega,\sk)e^{\frac{i\eta z}{2\eps}\big(\sum\limits_{j=1}^p\frac{|\xi_j|^2}{\omega_j}-\sum\limits_{l=1}^q\frac{|\zeta|^2}{\omega'_l}\big)}\,.
\end{equation}
This solves the evolution equation
\begin{equation}\label{eqn:varphi_evol}
    \begin{aligned}
        \partial_z\varphi^\eps&=L^\eps_{p,q}\varphi^\eps_{p,q},\quad \varphi^\eps_{p,q}(0)=\hat{\phi}^\eps_{p,q}(0)\,,
    \end{aligned}
\end{equation}
where the operator $L^\eps_{p,q}$ is a modified version of the operator in~\cite[Eq. (48)]{bal2024complex} given by
\begin{equation}\nonumber
    \begin{aligned}
        L^\eps_{p,q}&=\frac{p+q}{2}L_\eta+\sum_{j=1}^p\sum_{l=1}^qL^{\eps,1}_{j,l}+\sum_{1\le j<j'\le p}L^{\eps,2}_{j,j'}+\sum_{p+1\le l<l'\le p+q}L^{\eps,2}_{l,l'}\,,\\
        L_\eta&=-\frac{C_0}{\eta^2},\quad C_0=\frac{(\sum_{j=1}^p\omega_j^2+\sum_{l=1}^q{\omega'_l}^2)}{p+q}\frac{R(0)}{4}\,,\\
        L^{\eps,1}_{j,l}\psi&=\frac{\omega_j\omega'_l}{4\eta^2}\int_{\mathbb{R}^d}\psi(\xi_j-k,\zeta_l-k)e^{\frac{i\eta z}{2\eps}\big(\frac{g(\xi_j,k)}{\omega_j}-\frac{g(\zeta_l,k)}{\omega'_l}\big)}\frac{\hat{R}(k)\mathrm{d}k}{(2\pi)^d},\quad 1\le j\le p, 1\le l\le q\,,\\
        L_{j,j'}^{\eps,2}\psi&=-\frac{\omega_j\omega_{j'}}{4\eta^2}\int_{\mathbb{R}^d}\psi(\xi_j-k,\xi_{j'}+k)e^{\frac{i\eta z}{2\eps}\big(\frac{g(\xi_j,k)}{\omega_j}+\frac{g(\xi_{j'},-k)}{\omega_{j'}}\big)}\frac{\hat{R}(k)\mathrm{d}k}{(2\pi)^d},\quad 1\le j<j'\le p\,,\\
        L_{l,l'}^{\eps,2}\psi&=-\frac{\omega'_l\omega'_{l'}}{4\eta^2}\int_{\mathbb{R}^d}\psi(\zeta_l-k,\zeta_{l'}+k)e^{-\frac{i\eta z}{2\eps}\big(\frac{g(\zeta_l,k)}{\omega'_l}+\frac{g(\zeta_{l'},-k)}{\omega'_{l'}}\big)}\frac{\hat{R}(k)\mathrm{d}k}{(2\pi)^d},\quad p+1\le l<l'\le p+q\,.
    \end{aligned}
\end{equation}
Let $\mathcal{M}_B(\mathbb{R}^{(p+q)d})$ denote the Banach space of finite signed measures on $\mathbb{R}^{(p+q)d}$ equipped with the TV norm. We denote by $\|\cdot\|$ the TV norm of bounded measures, as well as the operator norm of operators acting on such measures. The following regularity estimates can be easily verified.
\begin{lemma}\label{lemma:phi_estim}
The following bounds hold.
\begin{equation}\nonumber
\begin{aligned}
  \|\varphi^\eps(z)\|&=\|\hat{\phi}^\eps(z)\|\le \|\hat{\phi}^\eps(0)\|e^{\frac{R(0)z(p+q)^2\omega_{\max}^2}{8\eta^2}}\,,\\
  \|\langle v_j\rangle^2\varphi^\eps(z,v;\cdot)\|&=\|\langle v_j\rangle^2\hat{\phi}^\eps(z,v;\cdot)\|\le \|\langle v_j\rangle^2\hat{\phi}^\eps(0,v;\cdot)\|e^{\frac{z(p+q)^2\omega_{\max}^2}{2\eta^2}\int_{\mathbb{R}^d}\langle k\rangle^2\hat{R}(k)\frac{\mathrm{d}k}{(2\pi)^d}},\quad 1\le j\le p+q \,.   
\end{aligned}
   \end{equation}
\end{lemma}
Let $U^\eps_{j,l}$ be the solution operator to
\begin{equation}\nonumber
    \partial_z U^\eps_{j,l}=L^{\eps,1}_{j,l}U^{\eps}_{j,l},\quad U^\eps_{j,l}(0)=\mathbb{I}\,.
\end{equation}
We also denote by $U_\eta(z)=e^{-\frac{C_0z}{\eta^2}}$ the solution operator to 
\begin{equation}\nonumber
    \partial_zU_\eta=L_\eta U_\eta,\quad U_\eta(0)=\mathbb{I}\,.
\end{equation}
Let $U^\eps$ denote the solution operator to 
\begin{equation}\nonumber
    \partial_zU^\eps=L^{\eps}_{p,q}U^\eps,\quad U^\eps(0)=\mathbb{I}\,.
\end{equation}
Now, we briefly recall a set of pairings introduced in~\cite[Section 4.2]{bal2024complex}. Let $\gamma=(j,l)$ denote a pairing with $1\le j\le p, 1\le l\le q$. For $1\le m\le M=\sum_{n=1}^{p\wedge q}\binom{p}{n}\binom{q}{n}n!$, let $\Lambda_m$ denote the set of ordered pairings $(j,l)$ such that for any $\gamma,\gamma'\in\Lambda_m$, $\gamma_1\neq\gamma'_1, \gamma_2\neq\gamma'_2$. Also given a set $\Lambda_m$, let $\Bar{\Lambda}_m$ denote the set of all pairings such that for any $\gamma\in\Lambda_m$, there exists a pairing $\gamma'\in\Bar{\Lambda}_m$ such that either $\gamma_1=\gamma'_1$ or $\gamma_2=\gamma'_2$ but not both. 

We recall that $\omega_j=\omega_0+\eps\eta\Omega_j, \omega'_l=\omega_0+\eps\eta\Omega'_l$. Then the phases that appear in the operators of the type $L^{\eps,1}_{j,l}$ are of the form
\begin{equation}\nonumber
    \frac{g(\xi_j,k)}{\omega_j}-\frac{g(\zeta_l,k)}{\omega'_l}=-|k|^2\Big(\frac{1}{\omega_j}-\frac{1}{\omega'_l}\Big)+2k\cdot\Big(\frac{\xi_j}{\omega_j}-\frac{\zeta_l}{\omega'_l}\Big)=-\frac{\eps\eta|k|^2(\Omega'_{l}-\Omega_j)}{\omega_j\omega'_l}+2k\cdot\Big(\frac{\xi_j}{\omega_j}-\frac{\zeta_l}{\omega'_l}\Big)\approx 2k\cdot\Big(\frac{\xi_j}{\omega_j}-\frac{\zeta_l}{\omega'_l}\Big)\,.
\end{equation}
Similarly, let $\Bar{\omega}_{j,j'}=\frac{\omega_j+\omega_{j'}}{2}$. Then the phases that appear in the $L^{\eps,2}_{j,j'}$ are of the form
\begin{equation}\nonumber
     \frac{g(\xi_j,k)}{\omega_j}+\frac{g(\xi_{j'},-k)}{\omega_{j'}}=-|k|^2\Big(\frac{1}{\omega_j}+\frac{1}{\omega_{j'}}\Big)+2k\cdot\Big(\frac{\xi_j}{\omega_j}-\frac{\xi_{j'}}{\omega_{j'}}\Big)=-\frac{2|k|^2\Bar{\omega}_{j,j'}}{\omega_j\omega_{j'}}+2k\cdot\Big(\frac{\xi_j}{\omega_j}-\frac{\xi_{j'}}{\omega_{j'}}\Big)\,.
\end{equation}
So both these operators can be dealt with in a similar manner as in the single frequency case in~\cite{bal2024complex} after slight modifications. We use this to write down a decomposition for the solution operator $U^\eps$ to~\eqref{eqn:varphi_evol} with the appropriate modifications from~\cite[Theorem 4.2]{bal2024complex} presented in the Appendix.

\begin{theorem}\label{thm:N_pq}
\begin{equation}\nonumber
    U^\eps=N^\eps_{p,q}+E^\eps_{p,q}\,,
\end{equation}
where
\begin{equation}\nonumber
    N^\eps_{p,q}=U_\eta^{\frac{p+q}{2}}\sum_{m=1}^M\prod_{\gamma\in\Lambda_m}(U^\eps_\gamma-\mathbb{I})+U_\eta^{\frac{p+q}{2}}\,,
\end{equation}
and the error solves 
\begin{equation}\nonumber
    \partial_zE^\eps_{p,q}=L^\eps_{p,q}E^\eps_{p,q}+\mathcal{E}^\eps_{p,q}, \quad E^\eps_{p,q}(0)=0\,,
\end{equation}
with source
\begin{equation}\nonumber
    \mathcal{E}^\eps_{p,q}=U_\eta^{\frac{p+q}{2}}\sum_{m=1}^M\sum_{\gamma'\in\Bar{\Lambda}_m}L^{\eps,1}_{\gamma'}\prod_{\gamma\in\Lambda_m}(U^\eps_\gamma-\mathbb{I})+L^{\eps,2}N^\eps_{p,q}\,.
\end{equation}
Moreover, the error satisfies the bound
\begin{equation}\nonumber
   \sup_{0\le s\le z} \|E^\eps_{p,q}(s)\|\le c(p,q,z)\eps^{\frac13}\,.
\end{equation}
\end{theorem}
As in \cite{bal2024complex}, $\eps^{\frac13}$ may be replaced here and below by any $\eps^\alpha$ with $\alpha<\frac12$ when $d=1$ and $\alpha<1$ when $d\geq2$.
\begin{proof}
   That the error ${E}^\eps_{p,q}$ solves the evolution equation as above follows the same steps as in the proof of~\cite[Theorem 4.2]{bal2024complex}. We only have to verify that the source term $\mathcal{E}^\eps_{p,q}$ is small as an operator on $\mathcal{M}_B(\mathbb{R}^{(p+q)d})$ after integrating in $z$. These are shown in Propositions~\ref{prop:L_1_bound} and~\ref{prop:L_2_bound} in Appendix~\ref{sec:Appendix}.
\end{proof}

\subsection{Gaussian summation rule in all variables}
Note that the phase compensated moments $\Psi^\eps_{p,q}$~\eqref{eqn:Psi_pq_def} and $\varphi^\eps:=\varphi^\eps_{p,q}$ are related as
\begin{equation}\nonumber
    \Psi^\eps_{p,q}(z_0,\cdots,z_0,v;\omega,\sk)=\varphi^\eps_{p,q}(z_0,v;\omega,\sk)e^{\frac{R(0)z_0}{8\eta^2}(\sum_{j=1}^p\omega_j^2+\sum_{l=1}^q{\omega'_l}^2)}\,.
\end{equation}
From Lemma~\ref{lem:phase_comp_pq}, this gives the $p+q$th moment for different $z$s
\begin{equation}\nonumber
\begin{aligned}
    &\hat{\mu}^\eps_{p,q}(z_1,\cdots,z_p,z'_1,\cdots,z'_q,v;\omega,\sk)\\
    &=\varphi^\eps_{p,q}(z_{0},v;\omega,\sk)e^{-\frac{i\eta}{2\eps}\big(\sum_{j=1}^p\frac{z_j|\xi_j|^2}{\omega_j}-\sum_{l=1}^q\frac{z'_l|\zeta_l|^2}{\omega'_l}\big)}e^{-\frac{R(0)}{8\eta^2}(\sum_{j=1}^p(z_j-z_{0})\omega_j^2+\sum_{l=1}^q(z'_l-z_{0}){\omega'}^2_l)}\\
    &+E^\eps_2(z,v;\omega,\sk)e^{-\frac{R(0)}{8\eta^2}(\sum_{j=1}^pz_j\omega_j^2+\sum_{l=1}^qz'_l{\omega'}^2_l)}e^{-\frac{i\eta}{2\eps}\big(\sum_{j=1}^p\frac{z_j|\xi_j|^2}{\omega_j}-\sum_{l=1}^q\frac{z'_l|\zeta_l|^2}{\omega'_l}\big)}\,.
\end{aligned}
    \end{equation}
    The latter term is $O(\eps^\alpha)$, $\alpha\in(0,1)$ in the total variation sense from Lemma~\ref{lem:phase_comp_pq}. Using the decomposition in Theorem~\ref{thm:N_pq}, we have
    \begin{equation}\nonumber
    %\label{eqn:mu_hat_dec}
    \begin{aligned}
         &\hat{\mu}^\eps_{p,q}(z_1,\cdots,z_p,z'_1,\cdots,z'_q,v;\omega,\sk)=\\
         &N^\eps_{p,q}(z_{0})\hat{\varphi}^\eps(0,v;\omega,\sk)e^{-\frac{i\eta}{2\eps}\big(\sum_{j=1}^p\frac{z_j|\xi_j|^2}{\omega_j}-\sum_{l=1}^q\frac{z'_l|\zeta_l|^2}{\omega'_l}\big)}e^{-\frac{R(0)}{8\eta^2}(\sum_{j=1}^p(z_j-z_{0})\omega_j^2+\sum_{l=1}^q(z'_l-z_{0}){\omega'}^2_l)}+E^\eps_3(z,v)+O(\eps^\alpha)\,,
    \end{aligned}
    \end{equation}
    with $\|E^\eps_3\|\le c(p,q,z_0)\eps^{\frac13}$. Inverse Fourier transforming this eventually leads us to the Gaussian summation rule for the macroscopic wavefield $\upsilon^\eps(h,x;\Omega,\kappa)=u^\eps\big(z_0+\eps\eta h,\frac{r}{\eps}+\eta x;\omega_0+\eps\eta\Omega,\sk_0+\eps\kappa\big)$. 
    
    \begin{lemma}[Higher moments of the wavefield]\label{Lemma:hat_mu_pq_dec}
        Under the scaling~\eqref{eqn:scalings}, the $p+q$th moment of the wavefield $u^\eps$ in the Fourier domain can be decomposed as
        \begin{equation}\nonumber
            \hat\mu^\eps_{p,q}(z,x;\omega,\sk)=\mathbb{E}\prod_{j=1}^p\hat u^\eps(z,x;\omega_j,\sk_j)\prod_{l=1}^q\hat u^{\eps\ast}(z'_l,y_l;\omega'_l,\sk'_l)=\tilde{\mathcal{N}}^\eps_{p,q}(z,v;\omega,\sk)\Phi^\eps(z,v;\omega)+\mathcal{E}^\eps_{p,q}\,,
        \end{equation}
        where
        \begin{equation}\label{eqn:N_tilde_pq}
            \begin{aligned}
                \tilde{\mathcal{N}}^\eps_{p,q}=&\Big(\prod_{\gamma\in\Lambda_m}\big(\hat\mu^\eps_{1,1}(z_{0},z_{0},\xi_{\gamma_1},\zeta_{\gamma_2};\omega_{\gamma_1},\omega'_{\gamma_2},\sk_{\gamma_1},\sk'_{\gamma_2})-\hat\mu^\eps_{1,0}(z_{0},\xi_{\gamma_1};\omega_{\gamma_1},\sk_{\gamma_1})\hat\mu^{\eps\ast}_{1,0}(z_{0},\zeta_{\gamma_2};\omega'_{\gamma_2},\sk'_{\gamma_2})\big)\Big)\\
                &\prod_{j\neq\gamma_1,l\neq \gamma_2}\hat\mu^\eps_{1,0}(z_{0},\xi_j;\omega_j,\sk_j)\hat\mu^{\eps\ast}_{1,0}(z_{0},\zeta_l;\omega'_l,\sk'_l)+\prod_{j=1}^p\hat\mu^\eps_{1,0}(z_{0},\xi_j;\omega_j,\sk_j)\prod_{l=1}^q\hat\mu^{\eps\ast}_{1,0}(z_{0},\zeta_l;\omega'_l,\sk'_l)\,,
            \end{aligned}
        \end{equation}
        \begin{equation}\nonumber
            \begin{aligned}
                \Phi^\eps(z,v;\omega)&=e^{-\frac{i\eta}{2\eps}(\sum_{j=1}^p(z_j-z_{0})|\xi_j|^2/\omega_j-\sum_{l=1}^q(z'_l-z_{0})|\zeta_l|^2/\omega'_l)}e^{-\frac{R(0)}{8\eta^2}(\sum_{j=1}^p(z_j-z_{0})\omega_j^2+\sum_{l=1}^q(z'_l-z_{0}){\omega'_l}^2)}
            \end{aligned}
        \end{equation}
        and 
        \begin{equation}\nonumber
            \|\mathcal{E}_{p,q}^\eps(z,v;\omega,\sk)\|\le c(p,q,z,\omega,\sk)\eps^\frac13\,.
        \end{equation}
    \end{lemma}
    % To complete the result in the physical domain, we first define by $\mathcal{F}$ the following functional of $p+q+pq$ arguments:
    % \begin{equation}\label{eqn:calF}
    %     \mathcal{F}(f_1,\cdots,f_{pq},g_1,\cdots,g_{p},g'_1,\cdots,g'_{q})=\sum_{m=1}^M\prod_{\gamma\in\Lambda_m}f_{\gamma}\prod_{j\neq \gamma_1}g_j\prod_{l\neq \gamma_2}g'_l+\prod_{j=1}^pg_j\prod_{l=1}^qg'_l\,.
    % \end{equation}
    We then have the following estimate in the physical variables.
\begin{theorem}\label{thm:Gaussian_rule_phys}
 Let $h=(h_1,\cdots,h_p,h'_1,\cdots,h'_q)\in\mathbb{R}^{p+q}$, $x=(x_1,\cdots,x_p,y_1,\cdots,y_q)\in\mathbb{R}^{(p+q)d}$, $\Omega=(\Omega_1,\cdots,\Omega_p,\Omega'_1,\cdots,\Omega'_q)\in\mathbb{R}^{p+q}$, and $\kappa=(\kappa_1,\cdots,\kappa_p,\kappa'_1,\cdots,\kappa'_q)\in\mathbb{R}^{(p+q)d}$. For fixed $(z_0,r;\omega_0,\sk_0)$ define % \gb{?}\an{fixed it}
     \begin{equation}\label{eqn:m_eps_pq}
         m^\eps_{p,q}(h,x;\Omega,\kappa)=\mathbb{E}\prod_{j=1}^p\upsilon^\eps(h_j,x_j;\Omega_j,\kappa_j)\prod_{l=1}^q\upsilon^{\eps\ast}(h'_l,y_l;\Omega'_l,\kappa'_l)\,.
     \end{equation}
Let $h_{j,l}=h_j-h'_l$, $\tau_{j,l}=x_j-y_l$, $\Omega_{j,l}=\Omega'_l-\Omega_j$ and $\kappa_{j,l}=\kappa_j-\kappa_l$. We have
     \begin{equation}\label{eqn:m_eps_pq_decomp}
         \begin{aligned}
             m^\eps_{p,q}(h,x;\Omega,\kappa)=\mathcal{N}^\eps_{p,q}+\mathscr{E}_{p,q},
         \end{aligned}
     \end{equation}
     where
     \begin{equation}\nonumber
     \begin{aligned}
                  \mathcal{N}^\eps_{p,q}=\begin{cases}
                      0,\quad p\neq q\\
                      \sum_{\pi_{p}}\prod_{j=1}^pm^\eps_{1,1}(h_{j,\pi_p(j)},\tau_{j,\pi_p(j)};\Omega_{j,\pi_p(j)},\kappa_{j,\pi_p(j)})
                  \end{cases}\,.
                  \end{aligned}
     \end{equation}
     Here, $\pi_{p}$ denotes a permutation of integers $\{1,\cdots,p\}$ without replacement and $m^\eps_{1,1}$ is given by Proposition~\ref{prop:m_11_expn}. For $z_{\max}=\max\{z_j,z'_l\}, \omega_{\max}=\max\{\omega_j,\omega'_l\}$, the error $\mathscr{E}_{p,q}^\eps$ is bounded uniformly in $x$ as
        \begin{equation}\nonumber
         \|\mathscr{E}^\eps_{p,q}(z,x;\omega,\sk)\|_\infty\le c_1(p,q,z_{\max},\omega_{\max})(\eps^\frac13+e^{-\frac{c_2(z_{\max},\omega_{\max})}{\eta^2}})\,.
        \end{equation}
%        \an{this is modified to include only the diffusion setting, which brings an extra exponential in the error.}
     % Here, $\mathcal{F}$ is given by~\eqref{eqn:calF}, $m^\eps_{1,0}=m^{\eps\ast}_{0,1}$ is the solution to the first moment~\eqref{eqn:first_mom}, \gb{This has to be modified to diffusion setting. Do we write a remark somewhere regarding the kinetic regime?}
     % \begin{equation}
     %     \tilde{m}^\eps_{1,1}(h,\tau;\Omega,\kappa)=\int\limits_{\mathbb{R}^{2d}}\hat{\tilde{M}}^\eps_{1,1}(z_0,\zeta,\xi;\omega_0,\sk_0,\Omega,\kappa)e^{-\frac{ih|\xi|^2}{2\omega_0}}e^{i(r\cdot\zeta+\tau\cdot \xi)}\frac{\mathrm{d}\xi\mathrm{d}\zeta}{(2\pi)^{2d}}\,,
     % \end{equation}
     % and $\tilde{M}^\eps_{1,1}=M^\eps_{1,1}-M^\eps_{0,0}$ with $M^\eps_{1,1}$ solving the PDE~\eqref{eqn:M_11_eps} and $M^\eps_{0,0}$ is the solution to~\eqref{eqn:M_00_eps}.    
           \begin{proof}
         Inverse Fourier transforming the decomposition in Lemma~\ref{Lemma:hat_mu_pq_dec}, we have
         \begin{equation}\nonumber
         \begin{aligned}
         &m^\eps_{p,q}(z,r,h,x;\omega_0,\sk_0,\Omega,\kappa)\\
         &=\int_{\mathbb{R}^{(p+q)d}}\tilde{\mathcal{N}}^\eps_{p,q}(z,v;\omega,\sk)\Phi^\eps(z,v;\omega)e^{i[\sum_{j=1}^p\xi_j\cdot(\eps^{-1}r+\eta x_j)-\sum_{l=1}^q\zeta_l\cdot(\eps^{-1}r+\eta y_l)}]\frac{\mathrm{d}v}{(2\pi)^{(p+q)d}}+\tilde{\mathcal{E}}^\eps_{p,q}\,,    
         \end{aligned}
         \end{equation}
         where $\tilde{\mathcal{E}}^\eps_{p,q}$ is the inverse Fourier transform of $\mathcal{E}^\eps_{p,q}$ defined in the same manner. Due to the total variation bound in Lemma~\ref{Lemma:hat_mu_pq_dec}, $\tilde{\mathcal{E}}^\eps_{p,q}$ is uniformly bounded as in the statement of the Theorem. From~\eqref{eqn:mu10}, the Fourier transform $\|\hat{\mu}^\eps_{1,0}(z)\|\le \|\hat{\mu}^\eps_{1,0}(0)\|e^{-\frac{k_0^2R(0)z}{8\eta^2}}$. Due to the uniform boundedness of $\hat{\mu}^\eps_{1,1}$ in the total variation norm, the terms in~\eqref{eqn:N_tilde_pq} retaining first order moments are smaller that $c(p,q)e^{-\frac{k_0^2R(0)z}{8\eta^2}}$. So in the leading term, all the terms are exponentially small if $p\neq q$ and when $p=q$, we only have to consider the contributions when $|m|=q=p$.

         When $p=q$, for the leading part, we have
         \begin{equation}\label{eqn:m_pq_eps_2}
\begin{aligned}\sum_{\pi_p}\int_{\mathbb{R}^{(p+q)d}}&\prod_{j=1}^p\hat\mu^\eps_{1,1}(z_0,\xi_{j},\zeta_{\pi_p(j)};\omega_{j},\omega'_{\pi_p(j)},\sk_{j},\sk_{\pi_p(j)})\\
                &\times\exp\Big[{-\frac{i\eta^2}{2}\Big(\frac{h_{j}|\xi_{j}|^2}{\omega_{j}}-\frac{h'_{\pi_p(j)}|\zeta_{\pi_p(j)}|^2}{\omega'_{\pi_p(j)}}\Big)}\Big]e^{\frac{ir}{\eps}\cdot(\xi_{j}-\zeta_{\pi_p(j)})}e^{i\eta(x_{j}\cdot\xi_{\gamma_1}-y_{\gamma_2}\cdot\zeta_{\pi_p(j)})}\frac{\mathrm{d}v}{(2\pi)^{(p+q)d}}\,.
             \end{aligned}
         \end{equation}  
Due to the product structure, the integrals over the dual variables can be distributed in pairs. The rest of the derivation then parallels the proof in Proposition~\ref{prop:m_11_expn}. 
              \end{proof}
 \end{theorem}
   \begin{proof}[Proof of Theorem~\ref{thm:finite_dim_mom}]
    Note that all the statistical moments of the random vector $\Upsilon^\eps$ are given by moments for the form~\eqref{eqn:m_eps_pq}.  From Theorem~\ref{thm:Gaussian_rule_phys} and Corollary~\ref{coro:m_11_eps_limit}, we have
       \begin{equation}\nonumber
           \lim_{\eps\to 0}m^\eps_{p,q}=\begin{cases}
               0,\quad p\neq q\\
               \sum_{\pi_p}\prod_{j=1}^pm_{1,1}(h_{j,\pi_p(j)},\tau_{j,\pi_p(j)};\Omega_{j,\pi_p(j)},\kappa_{j,\pi_p(j)}),\quad p=q\,,
           \end{cases}
       \end{equation}
       with $m_{1,1}$ given by~\eqref{eqn:m_11}. The limiting moments of $\Upsilon^\eps$ are thus equal to the corresponding statistical moments of the random vector $\Upsilon$. Due to the Carleman criterion~\cite{billingsley2017probability}, $\Upsilon$ is uniquely described through its statistical moments, and we have that $\Upsilon^\eps\Rightarrow\Upsilon$ as $\eps\to 0$.
   \end{proof}
\section{Proof of stochastic continuity and tightness of the field}
\label{sec:tight}
We first show the following regularity estimate which will be used to prove tightness.
\begin{lemma}\label{lemma:M_eps_11_reg}
    Let $M^\eps_{1,1}(z,r,\omega_0,\sk_0;\tau,\Omega,\kappa)$ be the solution to the PDE~\eqref{eqn:M_11_eps}. Then for $p\ge 1$ and $(z,r,\omega_0,\sk_0,\Omega,\kappa)$ in bounded subsets of $\mathbb{R}^{3d+3}$,
    \begin{equation}\nonumber
\|\langle \xi\rangle^{2p} \hat{M}^\eps_{1,1}(z)\|\le C_p(\omega_0,z,\sk_0)\|\langle \xi\rangle^{2p} \hat{M}^\eps_{1,1}(0)\|\,,
    \end{equation}
    where the total variation norm is w.r.t $(\zeta,\xi)$, the dual variables to $(r,\tau)$ respectively. The constant $C_p$ is independent of $\eps$.
    \begin{proof}
        As in the proof of Lemma~\ref{lemma:M_11}, the Feynman-Kac formula allows us to represent the solution $M^\eps_{1,1}$ in the Fourier domain as 
        \begin{equation}\nonumber
            \hat{M}_{1,1}^\eps(z,\zeta,\xi)=\mathbb{E}[\hat{M}^\eps_{1,1}(0,\zeta,\chi^\eta(0))e^{i\int_0^z{V}(\zeta,\chi^\eta(s))\mathrm{d}s}|\chi^\eta(z)=\xi],\quad {V}(\zeta,\xi)=\frac{\Omega|\xi|^2}{2\omega_0}+\frac{\xi\cdot\zeta}{\omega_0}\,.
        \end{equation}
        This gives the bound
        \begin{equation}\nonumber
            |\hat{M}_{1,1}^\eps(z,\zeta,\xi)|\le\mathbb{E}[|\hat{M}^\eps_{1,1}(0,\zeta,\chi^\eta(0))||\chi^\eta(z)=\xi]:=\hat{N}^\eps_{1,1}(z,\zeta,\xi)\,.
        \end{equation}
        The right hand side is the solution to the equation
        \begin{equation}\nonumber
            \begin{aligned}
        \partial_z\hat{N}^\eps_{1,1}&=\frac{\omega_0^2}{4\eta^2}\int_{\mathbb{R}^d}\hat{R}(k)[\hat{N}^\eps_{1,1}(z,\zeta,\xi-\eta k)-\hat{N}^\eps_{1,1}(z,\zeta,\xi)]\frac{\mathrm{d}k}{(2\pi)^d}\\
                \hat{N}^\eps_{1,1}(0,\cdot)&=|\hat{M}^\eps_{1,1}(0,\cdot)|\,.
            \end{aligned}
        \end{equation}
     Note that $\hat{N}^\eps_{1,1}\ge 0$ and $\|\hat{N}^\eps_{1,1}(z,\cdot)\|=\|\hat{N}^\eps_{1,1}(0,\cdot)\|$, so $\hat{N}^\eps_{1,1}$ is a probability density (after normalization). 
Suppose $p=1$.     This gives us the bound
        \begin{equation}\nonumber
            \|\langle\xi\rangle^2\hat{N}^\eps_{1,1}(z,\cdot)\|=(2\pi)^d[N^\eps_{1,1}(z,r,\tau)-\partial^2_\tau N^\eps_{1,1}(z,r,\tau)]|_{(r=0,\tau=0)}\,.
        \end{equation}
        Moreover, we have an explicit expression in the physical domain as
        \begin{equation}\nonumber
            N^\eps_{1,1}(z,r,\tau)=N^\eps_{1,1}(0,r,\tau)e^{\frac{\omega_0^2z[R(\eta\tau)-R(0)]}{4\eta^2}}\,.
        \end{equation}
        Taking derivatives w.r.t $\tau$ and noting that $\partial_\tau R(0)=0$, we have
                \begin{equation}\nonumber
            \|\langle\xi\rangle^2\hat{N}^\eps_{1,1}(z,\cdot)\|=(2\pi)^d[-\partial^2_\tau N^\eps_{1,1}(0,r,\tau)+\frac{\omega_0^2z}{4}\Delta R(0)N^\eps_{1,1}(0,r,\tau)]|_{(r=0,\tau=0)}\le C\|\langle\xi\rangle^2\hat{M}^\eps_{1,1}(0)\|\,.
        \end{equation}
        This finishes the proof for $p=1$. The case for arbitrary $p\ge 1$ can be dealt with in a similar manner.
    \end{proof}
\end{lemma}
\subsection{Tightness in propagation distance}
\begin{lemma}\label{lemma:tightness_z}
For $\omega=\omega_0+\eps\eta\Omega$, $\sk=\sk_0+\eps\kappa$, lateral location $\frac{r}{\eps}+\eta x$  fixed, and $h_1,h_2\in\mathbb{R}$ with $|h_1-h_2|< 1$,
\begin{equation}\nonumber
\mathbb{E}|\upsilon^\eps(h_1,x;\Omega,\kappa)-\upsilon^\eps(h_2,x;\Omega,\kappa)|^{2n}\le c(n,d)|h_1-h_2|^{n}\,.
    \end{equation}
    \begin{proof}
Here we look at differences of the form
\begin{equation}\nonumber
    \mathbb{E}|\upsilon^\eps(h_1,x;\Omega,\kappa)-\upsilon^\eps(h_2,x;\Omega,\kappa)|^{2n}=\mathbb{E}|u^\eps(z_0+\eps\eta  h_1;\eps^{-1}r+\eta x,\omega,\sk)-u^\eps(z_0+\eps\eta  h_2;\eps^{-1}r+\eta x,\omega,\sk)|^{2n},
\end{equation}
where $\omega=\omega_0+\eps\eta\Omega$, $\sk=\sk_0+\eps\kappa$ and lateral location $\frac{r}{\eps}+\eta x$ are fixed. We can then write this in terms of the phase compensated field~\eqref{eqn:psi_def} as
\begin{equation}\nonumber
    \begin{aligned}
 &\upsilon^\eps(h_1,x;\Omega,\kappa)-\upsilon^\eps(h_2,x;\Omega,\kappa)=\int_{\mathbb{R}^d} [\psi^\eps(z_0+\eps\eta h_1,\xi;\omega,\sk)e^{-\frac{i\eta(z_0+\eps\eta h_1)|\xi|^2}{2\eps\omega}}e^{-\frac{R(0)(z_0+\eps\eta h_1)\omega^2}{8\eta^2}}\\
 &-\psi^\eps(z_0+\eps\eta h_2,\xi;\omega,\sk)e^{-\frac{i\eta(z_0+\eps\eta h_2)|\xi|^2}{2\eps\omega}}e^{-\frac{R(0)(z_0+\eps\eta h_2)\omega^2}{8\eta^2}}]e^{i\xi\cdot(\eps^{-1}r+\eta x)}\frac{\mathrm{d}\xi}{(2\pi)^d}      \\
 &=\int_{\mathbb{R}^d}[\psi^\eps(z_0+\eps\eta h_1,\xi;\omega,\sk)-\psi^\eps(z_0+\eps\eta h_2,\xi;\omega,\sk)]e^{-\frac{i\eta(z_0+\eps\eta h_1)|\xi|^2}{2\eps\omega}}e^{-\frac{R(0)(z_0+\eps\eta h_1)\omega^2}{8\eta^2}}e^{i\xi\cdot(\eps^{-1}r+\eta x)}\frac{\mathrm{d}\xi}{(2\pi)^d}\\
 &+\int_{\mathbb{R}^d}\psi^\eps(z_0+\eps\eta h_2,\xi;\omega,\sk)e^{-\frac{i\eta(z_0+\eps\eta h_1)|\xi|^2}{2\eps\omega}}[e^{-\frac{R(0)(z_0+\eps\eta h_1)\omega^2}{8\eta^2}}-e^{-\frac{R(0)(z_0+\eps\eta h_2)\omega^2}{8\eta^2}}]e^{i\xi\cdot(\eps^{-1}r+\eta x)}\frac{\mathrm{d}\xi}{(2\pi)^d} \\
& \int_{\mathbb{R}^d}\psi^\eps(z_0+\eps\eta h_2,\xi;\omega,\sk)e^{-\frac{R(0)(z_0+\eps\eta h_2)\omega^2}{8\eta^2}}[e^{-\frac{i\eta(z_0+\eps\eta h_1)|\xi|^2}{2\eps\omega}}-e^{-\frac{i\eta(z_0+\eps\eta h_2)|\xi|^2}{2\eps\omega}}]e^{i\xi\cdot(\eps^{-1}r+\eta x)}\frac{\mathrm{d}\xi}{(2\pi)^d},%:=I_1+I_2+I_3\,. 
    \end{aligned}
\end{equation}
which we write as $I_1+I_2+I_3$ with obvious notation.
So we have
\begin{equation}\nonumber
 \mathbb{E}|\upsilon^\eps(h_1,x;\Omega,\kappa)-\upsilon^\eps(h_2,x;\Omega,\kappa)|^{2n}=\mathbb{E}|I_1+I_2+I_3|^{2n}\le C(n)(\mathbb{E}|I_1|^{2n}+\mathbb{E}|I_2|^{2n}+\mathbb{E}|I_3|^{2n})\,.
\end{equation}
For the first term, we have
\begin{equation}\nonumber
    \begin{aligned}
        \mathbb{E}|I_1|^{2n}&=\int_{\mathbb{R}^{2nd}}e^{-\frac{i\eta}{2\eps\omega}(z_0+\eps\eta h_1)\sum_{j=1}^{n}(\xi_j-\zeta_j)}e^{-\frac{nR(0)(z_0+\eps\eta h_1)\omega^2}{4\eta^2}}e^{i\sum_{j=1}^n(\xi_j-\zeta_j)\cdot (\eps^{-1}r+\eta x)}\\
        &\mathbb{E}\prod_{j=1}^n[\psi^\eps(z_0+\eps\eta h_1,\xi_j;\cdot)-\psi^\eps(z_0+\eps\eta h_2,\xi_j;\cdot)][\psi^{\eps\ast}(z_0+\eps\eta h_1,\zeta_j;\cdot)-\psi^{\eps\ast}(z_0+\eps\eta h_2,\zeta_j;\cdot)]\frac{\mathrm{d}v}{(2\pi)^{2nd}}\\
        &\le \|\mathbb{E}\prod_{j=1}^n[\psi^\eps(z_0+\eps\eta h_1,\xi_j;\cdot)-\psi^\eps(z_0+\eps\eta h_2,\xi_j;\cdot)][\psi^{\eps\ast}(z_0+\eps\eta h_1,\zeta_j;\cdot)-\psi^{\eps\ast}(z_0+\eps\eta h_2,\zeta_j;\cdot)]\|e^{-\frac{nR(0)(z_0+\eps\eta h_1)\omega^2}{4\eta^2}}\,.
    \end{aligned}
\end{equation}
To control this difference, we use Duhamel expansions. Without loss of generality, let $h_1\ge h_2$ so that from~\eqref{eqn:psi_exp},
\begin{equation}\nonumber
    \begin{aligned}
       &\delta w(^\eps\xi):=\psi^\eps(z_0+\eps\eta h_1,\xi;\cdot)-\psi^\eps(z_0+\eps\eta h_2,\xi;\cdot)\\
       &\sum_{n\ge 1}\Big(\frac{\omega}{2\eta}\Big)^{n}\int_{[z_0+\eps\eta h_2,z_0+\eps\eta h_1]^n_<}\int_{\mathbb{R}^{nd}}\psi^\eps(z_0+\eps\eta h_2,\xi-k_1-\cdots-k_n;\cdot)e^{\frac{i\eta}{2\eps\omega}[s_1g(\xi,k_1)+s_2g(\xi-k_1,k_2)+\cdots]}e^{\frac{in\pi}{2}}\prod_{j=1}^n\frac{\mathrm{d}\hat{B}(s_j,k_j)}{(2\pi)^d}\,.
    \end{aligned}
\end{equation}
So a product of $2n$ such (complex-symmetrized) differences is given by
\begin{equation}\nonumber
    \begin{aligned}
        \mathbb{E}\prod_{j=1}^n\delta w^\eps(\xi_j)\delta w^{\eps\ast}(\zeta_j)&=\sum_{N\ge n}\sum_{n_1+\cdots+n_{2n}=2N, n_j\ge 1}\Big(\frac{\omega}{2\eta}\Big)^{2N}\int_{[z_0+\eps\eta h_2,z_0+\eps\eta h_1]^{n_1}_<}\cdots\int_{[z_0+\eps\eta h_2,z_0+\eps\eta h_1]^{n_{2n}}_<}\\
        &\int_{\mathbb{R}^{2Nd}}\Psi^\eps(z_2,\cdots,z_2,v-A_{\vec{n}}\vec{k})e^{\frac{i\eta}{2\eps\omega}G_{\vec{n}}(\vec{s},\vec{k},v)}\mathbb{E}[\prod_{j=1}^{2N}\frac{\mathrm{d}\hat{B}(s_j,k_j)}{(2\pi)^d}]\,,
    \end{aligned}
\end{equation}
where $z_2=z_0+\eps\eta h_2$, $\vec{n}=(n_1,\cdots,n_{2n})$, $\vec{k}=(k_1,\cdots,k_{2N})$, $\vec{s}=(s_1,\cdots,s_{2N})$, $A_{\vec{n}}$ is a linear operator and $G_{\vec{n}}$ is real-valued. This gives
\begin{equation}\nonumber
\begin{aligned}
    \|\mathbb{E}\prod_{j=1}^n\delta w^\eps(\xi_j)\delta w^{\eps\ast}(\zeta_j)\|&\le \|\Psi_{n,n}^\eps(z_2,\cdots,z_2)\|     \sum_{N\ge n}\sum_{n_1+\cdots+n_{2n}=2N, n_j\ge 1}\Big(\frac{\omega^2}{8\eta^2}\Big)^{N}(\eps\eta(h_1-h_2))^NR(0)^N\frac{(2N)!}{N!n_1!\cdots n_{2n}!}\,.
\end{aligned}
\end{equation}
The summation has powers higher than $(\eps/\eta)^n(h_1-h_2)^n$, which gives the bound
\begin{equation}\nonumber
    \begin{aligned}
      & \|\mathbb{E}\prod_{j=1}^n\delta w^\eps(\xi_j)\delta w^{\eps\ast}(\zeta_j)\|\\
      &\le \|\Psi_{n,n}^\eps(z_2,\cdots,z_2)\| \Big(\frac{\eps}{\eta}|h_1-h_2|\Big)^n    \sum_{N\ge n}\sum_{n_1+\cdots+n_{2n}=2N, n_j\ge 1}  (\omega^2R(0))^N\frac{(2N)!}{N!n_1!\cdots n_{2n}!}\\
      &\le C\Big(\frac{\eps}{\eta}|h_1-h_2|\Big)^n  \|\Psi_{n,n}^\eps(0,\cdots,0)\|e^{n^2z_2/\eta^2}\,.
    \end{aligned}
\end{equation}
In the last inequality, we have used the regularity bound $\|\Psi_{p,q}^\eps(z,\cdots,z)\|\le \|\Psi_{np,q}^\eps(0,\cdots,0)\|e^{(p+q)^2R(0)z/4\eta^2}$ and the combinatorial Lemma~\ref{lemma:fact_sum}. Due to the scaling relationship between $\eps$ and $\eta$, we have that 
\begin{equation}\nonumber
    \mathbb{E}|I_1|^{2n}\le c(n)|h_1-h_2|^n\,.
\end{equation}
The boundedness of $\mathbb{E}|I_2|^{2n}$ follows by noting that
\begin{equation}\nonumber
    |e^{-\frac{R(0)(z_0+\eps\eta h_1)\omega^2}{8\eta^2}}-e^{-\frac{R(0)(z_0+\eps\eta h_2)\omega^2}{8\eta^2}}|\le \frac{\eps |h_1-h_2|R(0)\omega^2}{8\eta}
\end{equation}
so that $2n$ such products give
\begin{equation}\nonumber
    \mathbb{E}|I_2|^{2n}\le \|\Psi_{n,n}^\eps(z_2,\cdots,z_2)\|e^{-\frac{nR(0)z_2\omega^2}{4\eta^2}}\Big(\frac{\eps |h_1-h_2|R(0)\omega^2}{8\eta}\Big)^{2n}\le C(n)|h_1-h_2|^{2n}\,.
\end{equation}
The third term $\mathbb{E}|I_3|^{2n}$ involves contributions from the Laplacians and requires more care. We have
\begin{equation}\nonumber
    \begin{aligned}
        \mathbb{E}|I_3|^{2n}=\int_{\mathbb{R}^{2nd}}&\Psi^\eps(z_2,\cdots,z_2,v;\cdot)e^{-\frac{nR(0)z_2\omega^2}{4\eta^2}}e^{-\frac{i\eta z_2}{2\eps\omega}\sum_{j=1}^n(|\xi_j|^2-|\zeta_j|^2)}e^{i\sum_{j=1}^n(\xi_j-\zeta_j)\cdot(\eps^{-1}r+\eta x)}\\
        &\prod_{j=1}^n(e^{-\frac{i\eta^2}{2\omega}(h_1-h_2)|\xi_j|^2}-1)(e^{\frac{i\eta^2}{2\omega}(h_1-h_2)|\zeta_j|^2}-1)\frac{\mathrm{d}v}{(2\pi)^{2nd}}\,.
    \end{aligned}
\end{equation}

In the kinetic regime with $\eta=1$, bounding the above by $|h_1-h_2|^{2n}$ is straightforward due to the boundedness from Lemma~\ref{lemma:phi_estim}:
\begin{equation}\nonumber
\|\prod_{j=1}^{2n}\langle v_j\rangle^2\Psi^\eps_{n,n}(z_2,\cdots,z_2,v)\|e^{-\frac{nR(0)z_2\omega^2}{4\eta^2}}\le \|\prod_{j=1}^{2n}\langle v_j\rangle^2\Psi^\eps_{n,n}(0,\cdots,0,v)\|e^{cn^2}\,.    
\end{equation}
In the diffusive regime, this grows as $e^{cn^2/\eta^2}$, which is not compensated by $|h_1-h_2|^{2n}$ unless the difference $|h_1-h_2|$ is of $\mathcal{O}(\eps^\alpha)$ for some $0<\alpha<1$. However, when $\eta\ll 1$ with $|h_1-h_2|>O(\eps^\alpha)$ the complex Gaussian approximation holds which can be leveraged as in~\cite{bal2024complex}. We always have the first type of bound
\begin{equation}\label{eqn:E_I3_bound1}
    \mathbb{E}|I_3|^{2n}\le \|\varphi^\eps(0)\|\eta^{2n}|h_1-h_2|^{2n}e^{cn^2/\eta^2}\,,
\end{equation}
which we can use when $|h_1-h_2|=\mathcal{O}(\eps^\alpha)$. When $|h_1-h_2|$ is large, we use the complex Gaussian approximation. In this case, using Lemma~\ref{Lemma:hat_mu_pq_dec} with $z_j=z'_l=z_2$ we have
\begin{equation}\nonumber
    \begin{aligned}
        \mathbb{E}|I_3|^{2n}&=\int_{\mathbb{R}^{2nd}}\tilde{\mathcal{N}}_{n,n}^\eps(z_2,v;\cdot)e^{i\sum_{j=1}^n(\xi_j-\zeta_j)\cdot(\eps^{-1}r+\eta x)}\prod_{j=1}^n\delta\hat{G}(\eta^2(h_1-h_2),\xi_j)\delta\hat{G}^\ast(\eta^2(h_1-h_2),\zeta_j)\frac{\mathrm{d}v}{(2\pi)^{2nd}}+\mathcal{O}(\eps^{\frac13})\,,
    \end{aligned}
\end{equation}
for $\delta \hat{G}(h,\xi)=\hat{G}(h,\xi)-1$ with $G(h,x)=\big(\frac{\omega}{2\pi ih}\big)^{d/2}e^{\frac{i\omega|x|^2}{2 h}}$. Expanding $\tilde{N}^\eps_{n,n}$  gives the leading term to be
\begin{equation}\nonumber
    \begin{aligned}
       & \sum_{m=1}^M\int_{\mathbb{R}^{2nd}}\prod_{\gamma\in\Lambda_m}[\hat{\mu}^\eps_{1,1}(z_2,z_2,\xi_{\gamma_1},\zeta_{\gamma_2};\cdot)-\hat{\mu}^\eps_{1,0}(z_2,\xi_{\gamma_1};\cdot)\hat{\mu}^{\eps\ast}_{1,0}(z_2,\zeta_{\gamma_2};\cdot)]\delta\hat{G}(\eta^2(h_1-h_2),\xi_{\gamma_1})\delta\hat{G}^\ast(\eta^2(h_1-h_2),\zeta_{\gamma_2})\\
        &\prod_{j\neq \gamma_1}\hat{\mu}^\eps_{1,0}(z_2,\xi_j;\cdot)\delta\hat{G}(\eta^2(h_1-h_2),\xi_j)\prod_{l\neq \gamma_2}\hat{\mu}^{\eps\ast}_{1,0}(z_2,\zeta_l;\cdot)\delta\hat{G}^\ast(\eta^2(h_1-h_2),\zeta_j)e^{i\sum_{j=1}^n(\xi_j-\zeta_j)\cdot(\eps^{-1}r+\eta x)}\frac{\mathrm{d}v}{(2\pi)^{2nd}}\\
        &+\int_{\mathbb{R}^{2nd}}\prod_{j=1}^p\hat{\mu}^\eps_{1,0}(z_2,\xi_j;\cdot)\delta\hat{G}(\eta^2(h_1-h_2),\xi_j)\prod_{l=1}^q\hat{\mu}^{\eps\ast}_{1,0}(z_2,\zeta_l;\cdot)\delta\hat{G}^\ast(\eta^2(h_1-h_2),\zeta_j)e^{i\sum_{j=1}^n(\xi_j-\zeta_j)\cdot(\eps^{-1}r+\eta x)}\frac{\mathrm{d}v}{(2\pi)^{2nd}}
    \end{aligned}
\end{equation}
The product structure of the integrals allows us to distribute the integrals over the pairings $\gamma$. The terms involving first moments decay exponentially, with
\begin{equation}\nonumber
    \|\langle\xi\rangle^2\hat{\mu}^\eps_{1,0}(z,\xi;\omega,\sk)\|\le\|\langle\xi\rangle^2\hat{\mu}^\eps_{1,0}(0,\xi;\omega,\sk)\|e^{-\frac{R(0)\omega^2 z}{8\eta^2}}\,. 
\end{equation}
This gives
\begin{equation}\nonumber
    \begin{aligned}
        \mathbb{E}|I_3|^{2n}&\le \sum_{m=1}^M\eta^{4n-2|m|}|h_1-h_2|^{2n-|m|}\|\langle\xi\rangle^2\hat{\mu}^\eps_{1,0}(0,\xi;\omega,\sk)\|^{2n-|m|}e^{-\frac{(2n-|m|)R(0)\omega^2 z}{8\eta^2}}\\
        &\Big(\int_{\mathbb{R}^{2d}}\hat{\mu}^\eps_{1,1}(z_2,z_2,\xi,\zeta;\cdot)\delta\hat{G}(\eta^2(h_1-h_2),\xi)\delta\hat{G}^\ast(\eta^2(h_1-h_2),\zeta)e^{i(\xi-\zeta)\cdot(\eps^{-1}r+\eta x)}\frac{\mathrm{d}\xi\mathrm{d}\zeta}{(2\pi)^{2d}}\Big)^{|m|}\\
        &+\Big(\eta^{4}|h_1-h_2|^{2}\|\langle\xi\rangle^2\hat{\mu}^\eps_{1,0}(0,\xi;\omega,\sk)\|^{2}e^{-\frac{R(0)\omega^2 z}{4\eta^2}}\Big)^{n}+\mathcal{O}(\eps^{\frac13})\,.
    \end{aligned}
\end{equation}
So it suffices to look at the contribution from the second moments. In particular, let $h=h_1-h_2$ and consider
\begin{equation}\nonumber
    \begin{aligned}
     &   \int_{\mathbb{R}^{2d}}\hat{\mu}^\eps_{1,1}(z_2,z_2,\xi,\zeta;\cdot)(1-e^{-\frac{i\eta^2h|\xi|^2}{2\omega}})(1-e^{\frac{i\eta^2h|\zeta|^2}{2\omega}})e^{i(\eps^{-1}r+\eta x)\cdot(\xi-\zeta)}\frac{\mathrm{d}\xi\mathrm{d}\zeta}{(2\pi)^{2d}}\\
     &= \int_{\mathbb{R}^{4d}}{\mu}^\eps_{1,1}(z_2,z_2,x',y';\cdot)e^{-i(\xi\cdot x'-\zeta\cdot y')}(1-e^{-\frac{i\eta^2h|\xi|^2}{2\omega}})(1-e^{\frac{i\eta^2h|\zeta|^2}{2\omega}})e^{i(\eps^{-1}r+\eta x)\cdot(\xi-\zeta)}\frac{\mathrm{d}x'\mathrm{d}y'\mathrm{d}\xi\mathrm{d}\zeta}{(2\pi)^{2d}}\,.
    \end{aligned}
\end{equation}
From the change of variables $(x',y')\to\big(\frac{r'}{\eps}+\frac{\eta\tau}{2},\frac{r'}{\eps}-\frac{\eta\tau}{2}\big)$, $(\xi,\zeta)\to\big(\frac{\xi}{\eta}+\frac{\eps\zeta}{2}, \frac{\xi}{\eta}-\frac{\eps\zeta}{2}\big)$ and using Proposition~\ref{prop:m_11_expn} (with $h=\Omega=0$), we have that the above integral upto $O(\eps^\frac13)$ is
\begin{equation}\nonumber
    \begin{aligned}
     & \int_{\mathbb{R}^{4d}}M^\eps_{1,1}(z_2,r',\tau';\Omega=0)e^{-i(r'\cdot\zeta -\tau'\cdot\xi)}(1-e^{-\frac{i\eta^2h|\frac{\xi}{\eta}+\frac{\eps\zeta}{2}|^2}{2\omega}})(1-e^{\frac{i\eta^2h|\frac{\xi}{\eta}-\frac{\eps\zeta}{2}|^2}{2\omega}})e^{i(r+\eps\eta x)\cdot\zeta}\frac{\mathrm{d}r'\mathrm{d}\tau'\mathrm{d}\xi\mathrm{d}\zeta}{(2\pi)^{2d}}  \\
      &=\int_{\mathbb{R}^{2d}}\hat{M}^\eps_{1,1}(z_2,\zeta,\xi;\Omega=0)(1-e^{-\frac{i\eta^2h|\frac{\xi}{\eta}+\frac{\eps\zeta}{2}|^2}{2\omega}})(1-e^{\frac{i\eta^2h|\frac{\xi}{\eta}-\frac{\eps\zeta}{2}|^2}{2\omega}})e^{i(r+\eps\eta x)\cdot\zeta}\frac{\mathrm{d}\xi\mathrm{d}\zeta}{(2\pi)^{2d}}\,,
    \end{aligned}
\end{equation}
where $M^\eps_{1,1}$ follows equation~\eqref{eqn:M_11_eps}. Upto $O(\eps^\frac13)$, the above integral can be written as
\begin{equation}\nonumber
    \begin{aligned}
        \int_{\mathbb{R}^{2d}}\hat{M}^\eps_{1,1}(z_2,\zeta,\xi;\Omega=0)(1-e^{-\frac{ih|\xi|^2}{2\omega}})(1-e^{\frac{ih|\xi|^2}{2\omega}})e^{ir\cdot\zeta}\frac{\mathrm{d}\xi\mathrm{d}\zeta}{(2\pi)^{2d}}\,,
    \end{aligned}
\end{equation}
which using Lemma~\ref{lemma:M_eps_11_reg} is bounded by
\begin{equation}\nonumber
   \frac{|h|^2}{\omega^2} \int_{\mathbb{R}^{2d}}|\xi|^4\hat{M}^\eps_{1,1}(z_2,\zeta,\xi;\Omega=0)\frac{\mathrm{d}\xi\mathrm{d}\zeta}{(2\pi)^{2d}}\le C|h|^2\,.
\end{equation}
This means $\mathbb{E}|I_3|^{2n}\le C\min\{\eps^\frac13+|h|^{2n},|h|^{2n}e^{c/\eta^2}\}$.
Now set $h_0=\eps^\frac{1}{6n}$. For $|h|>h_0$, we choose the first bound and for $|h|<h_0$, we use the second bound. In the latter case, we have for any $\alpha\in(0,1)$, $|h|^{2n}e^{c/\eta^2}\le |h|^{2n(1-\alpha)}\eps^{\alpha/3}e^{c/\eta^2}\le C(\alpha,n)|h|^{2n(1-\alpha)}$. Either way, the influence of the difference in the phase compensated fields from $\mathbb{E}|I_1|^{2n}$ brings in a bound of $|h|^{n}$, limiting the exponent of $h$ to $n$. This finishes the proof of the Lemma.
    \end{proof}
\end{lemma}

\subsection{Tightness in frequency}
\begin{lemma}\label{lemma:tightness_omega}
For $z=z_0+\eps\eta h$, $\sk=\sk_0+\eps\kappa$, lateral location $\frac{r}{\eps}+\eta x$  fixed, and $\Omega_1,\Omega_2\in\mathbb{R}$ with $|\Omega_1-\Omega_2|< 1$,
\begin{equation}\nonumber
 \mathbb{E}|\upsilon^\eps(h,x;\Omega_1,\kappa)-\upsilon^\eps(h,x;\Omega_2,\kappa)|^{2n}\le c_\alpha(n,d)|\Omega_1-\Omega_2|^{2n\alpha}\,,
    \end{equation}
    where $\alpha$ is any real number in $(0,1)$.
    \begin{proof}
Following the phase compensation approach as above, for fixed $z=z_0+\eps\eta h$, $\eps^{-1}r+\eta x$, $\sk=\sk_0+\eps\kappa$,
\begin{equation}\nonumber
    \begin{aligned}
      & u^\eps(z,\eps^{-1}r+\eta x;\omega_1,\sk)- u^\eps(z,\eps^{-1}r+\eta x;\omega_2,\sk)=\int_{\mathbb{R}^d}[ \hat{u}^\eps(z,\xi;\omega_1,\sk)- \hat{u}^\eps(z,\xi;\omega_2,\sk)]e^{i\xi\cdot(\eps^{-1}r+\eta x)}\frac{\mathrm{d}\xi}{(2\pi)^d}\\
       &=\int_{\mathbb{R}^d}[\psi^\eps(z,\xi;\omega_1,\cdot)e^{-\frac{iz\eta|\xi|^2}{2\eps\omega_1}}e^{-\frac{R(0)z\omega_1^2}{8\eta^2}}-\psi^\eps(z,\xi;\omega_2,\cdot)e^{-\frac{iz\eta|\xi|^2}{2\eps\omega_2}}e^{-\frac{R(0)z\omega_2^2}{8\eta^2}}]e^{i\xi\cdot(\eps^{-1}r+\eta x)}\frac{\mathrm{d}\xi}{(2\pi)^d}\\
       &=\int_{\mathbb{R}^d}[\psi^\eps(z,\xi;\omega_1,\cdot)-\psi^\eps(z,\xi;\omega_2,\cdot)]e^{-\frac{iz\eta|\xi|^2}{2\eps\omega_1}}e^{-\frac{R(0)z\omega_1^2}{8\eta^2}}\frac{\mathrm{d}\xi}{(2\pi)^d}\\
       &+\int_{\mathbb{R}^d}\psi^\eps(z,\xi;\omega_2,\cdot)e^{-\frac{iz\eta|\xi|^2}{2\eps\omega_1}}(e^{-\frac{R(0)z\omega_1^2}{8\eta^2}}-e^{-\frac{R(0)z\omega_2^2}{8\eta^2}})\frac{\mathrm{d}\xi}{(2\pi)^d}\\
       &+\int_{\mathbb{R}^d}\psi^\eps(z,\xi;\omega_2,\cdot)e^{-\frac{R(0)z\omega_2^2}{8\eta^2}}(e^{-\frac{iz\eta|\xi|^2}{2\eps\omega_1}}-e^{-\frac{iz\eta|\xi|^2}{2\eps\omega_2}})\frac{\mathrm{d}\xi}{(2\pi)^d}:=J_1+J_2+J_3\,.
    \end{aligned}
\end{equation}
Let $\Omega=\Omega_2-\Omega_1$. Note that
\begin{equation}\nonumber
    e^{-\frac{R(0)z\omega_1^2}{8\eta^2}}-e^{-\frac{R(0)z\omega_2^2}{8\eta^2}}=e^{-\frac{R(0)z\omega_1^2}{8\eta^2}}(1-e^{-\frac{R(0)\eps\eta\Omega(\omega_1+\omega_2)}{8}})
\end{equation}
and
\begin{equation}\nonumber
   e^{-\frac{iz\eta|\xi|^2}{2\eps\omega_1}}-e^{-\frac{iz\eta|\xi|^2}{2\eps\omega_2}}= e^{-\frac{iz\eta|\xi|^2}{2\eps\omega_1}}(1-e^{\frac{iz\eta^2\Omega|\xi|^2}{2\omega_1\omega_2}})\,.
\end{equation}
So as in the proof of Lemma~\ref{lemma:tightness_z}, terms $J_2$ and $J_3$ can be shown to contribute
\begin{equation}\nonumber
    \mathbb{E}|J_2|^{2n}\le C_1(n,d)|\Omega|^{2n},\quad \mathbb{E}|J_3|^{2n}\le C_{2,\alpha}(n,d)|\Omega|^{2n(1-\alpha)}\,,
\end{equation}
where $\alpha$ is an arbitrary number in $(0,1)$.
Expanding the phase compensated field gives for the difference in the first term $J_1$,
\begin{equation}\nonumber
    \begin{aligned}
        &\psi^\eps_n(z,\xi;\omega_1)-\psi^\eps_n(z,\xi;\omega_2)=\Big(\frac{i\omega_1}{2\eta}\Big)^{n}\int_{[0,z]^n_<}\int_{\mathbb{R}^{nd}}\psi_0^\eps(\xi-k_1-\cdots-k_n)e^{\frac{i\eta}{2\eps\omega_1}[s_1g(\xi,k_1)+\cdots]}\prod\frac{\mathrm{d}\hat{B}(s_j,k_j)}{(2\pi)^d}\\
        -&\Big(\frac{i\omega_2}{2\eta}\Big)^{n}\int_{[0,z]^n_<}\int_{\mathbb{R}^{nd}}\psi_0^\eps(\xi-k_1-\cdots-k_n)e^{\frac{i\eta}{2\eps\omega_2}[s_1g(\xi,k_1)+\cdots]}\prod\frac{\mathrm{d}\hat{B}(s_j,k_j)}{(2\pi)^d}\\
        &=\Big(\frac{i\omega_1}{2\eta}\Big)^{n}\Big(1-\frac{\omega_2^n}{\omega_1^n}\Big)\int_{[0,z]^n_<}\int_{\mathbb{R}^{nd}}\psi_0^\eps(\xi-k_1-\cdots-k_n)e^{\frac{i\eta}{2\eps\omega_1}[s_1g(\xi,k_1)+\cdots]}\prod\frac{\mathrm{d}\hat{B}(s_j,k_j)}{(2\pi)^d}\\
        &+\Big(\frac{i\omega_2}{2\eta}\Big)^{n}\int_{[0,z]^n_<}\int_{[0,z]^n_<}\int_{\mathbb{R}^{nd}}\psi_0^\eps(\xi-k_1-\cdots-k_n)(e^{\frac{i\eta}{2\eps\omega_1}[s_1g(\xi,k_1)+\cdots]}-e^{\frac{i\eta}{2\eps\omega_2}[s_1g(\xi,k_1)+\cdots]})\prod\frac{\mathrm{d}\hat{B}(s_j,k_j)}{(2\pi)^d}\\
        &:=\tilde{J}_{1,1}+\tilde{J}_{1,2}\,.
    \end{aligned}
\end{equation}
For $\omega_2-\omega_1=\eps\eta\Omega$, 
\begin{equation}\nonumber
 \Big|1-\frac{\omega_2^n}{\omega_1^n}\Big|=\Big|1-\Big(\frac{\omega_1+\eps\eta\Omega}{\omega_1}\Big)^n\Big|\le \frac{\eps\eta|\Omega|}{\omega_1}\sum_{j=1}^n\binom{n}{j}\le    2^n\frac{\eps\eta|\Omega|}{\omega_1}\,.
\end{equation}
Also, the difference in phases in $\tilde{J}_{1,2}$ is of the form
\begin{equation}\nonumber
  \begin{aligned}
      e^{\frac{i\eta}{2\eps\omega_1}[s_1g(\xi,k_1)+\cdots]}-e^{\frac{i\eta}{2\eps\omega_2}[s_1g(\xi,k_1)+\cdots]}=e^{\frac{i\eta}{2\eps\omega_1}[s_1g(\xi,k_1)+\cdots]}(1-e^{\frac{i\eta^2\Omega}{2\omega_1\omega_2}[s_1g(\xi,k_1)+\cdots]})\,.
  \end{aligned}  
\end{equation}
For $1\le j\le n-1$, we have
\begin{equation}\nonumber
    |g(\xi-k_1-\cdots-k_j,k_{j+1})|=|\xi-k_1-\cdots-k_j|^2-|\xi-k_1-\cdots-k_{j+1}|^2\le 2^{j+2}\langle \xi\rangle^2\prod_{l=1}^{j+1}\langle \le k_l\rangle^2\le 2^{n+2}\langle\xi\rangle^2\prod_{l=1}^{n}\langle k_l\rangle^2\,.
\end{equation}
This gives
\begin{equation}\nonumber
    |e^{\frac{i\eta}{2\eps\omega_1}[s_1g(\xi,k_1)+\cdots]}-e^{\frac{i\eta}{2\eps\omega_2}[s_1g(\xi,k_1)+\cdots]}|\le \frac{\eta^2|\Omega|}{\omega_1\omega_2}nz2^{n}\langle \xi\rangle^2\prod_{j=1}^{n}\langle k_j\rangle^2\,.
\end{equation}
Computing the contributions from a product of $2n$ complex symmetrized copies of $\tilde{J}_{1,j}$ as in Lemma~\ref{lem:phase_comp_pq} along with the estimates above gives
\begin{equation}\nonumber
    \mathbb{E}|J_1|^{2n}\le C(n,d)|\Omega|^{2n}e^{c/\eta^2}\,.
\end{equation}
As before, this is controlled in $\Omega$ only when $\eta$ is much larger that $|\Omega|$. For the case $|\Omega|\gg \eta$, we use the complex Gaussian approximation. This gives
\begin{equation}\nonumber   \mathbb{E}|\upsilon^\eps(h,x;\Omega_1,\kappa)-\upsilon^\eps(h,x;\Omega_2,\kappa)|^{2n}=n!\mathbb{E}|\upsilon^\eps(h,x;\Omega_1,\kappa)-\upsilon^\eps(h,x;\Omega_2,\kappa)|^2 + O(\eps^\frac13)\,.
\end{equation}
So it suffices to control differences in second moments. Note that
\begin{equation}\nonumber
\mathbb{E}|\upsilon^\eps(h,x;\Omega_1,\kappa)-\upsilon^\eps(h,x;\Omega_2,\kappa)|^2=E^\eps(\Omega)+O(\eps^\alpha), \quad\alpha\in(0,1)\,,
\end{equation}
where $E^\eps(\Omega)=2M^\eps_{1,1}(0)-M^\eps_{1,1}(\Omega)-M^\eps_{1,1}(-\Omega)$ and $M^\eps_{1,1}(\Omega)$ solves Eq.~\eqref{eqn:M_11_eps}. Then $E^\eps(\Omega)$ solves the PDE
\begin{equation}\nonumber
    \partial_zE^\eps=\frac{i\Omega}{2\omega_0^2}\Delta_\tau(M^\eps_{1,1}\big(\Omega)-M^\eps_{1,1}(-\Omega)\big) +\frac{i}{\omega_0}\partial_r\cdot\partial_\tau E^\eps+\frac{\omega_0^2}{4\eta^2}[R(\eta\tau)-R(0)]E^\eps,\quad E^\eps(z=0)=0\,.
\end{equation}
Controlling the source $\Delta_\tau(M^\eps_{1,1}\big(\Omega)-M^\eps_{1,1}(-\Omega)\big)$ in the total variation norm in the Fourier domain should yield that $E^\eps$ is controlled in the uniform sense in the physical domain. Note that from Lemma~\ref{lemma:M_eps_11_reg}, we have $\|\langle\xi\rangle^2\hat{M}^\eps_{1,1}(z,\zeta,\xi;\cdot)\|$ uniformly in $\Omega$. However, we lose a power of $|\Omega|$ in this process. This can be recovered by controlling the difference $\tilde{E}^\eps=M^\eps_{1,1}(\Omega)-M^\eps_{1,1}(-\Omega)$ further. We have
\begin{equation}\nonumber
    \partial_z\tilde{E}^\eps=\frac{i\Omega}{2\omega_0^2}\Delta_\tau(M^\eps_{1,1}\big(\Omega)-M^\eps_{1,1}(-\Omega)\big) +\frac{i}{\omega_0}\partial_r\cdot\partial_\tau\tilde{E}^\eps+\frac{\omega_0^2}{4\eta^2}[R(\eta\tau)-R(0)]\tilde{E}^\eps,\quad \tilde{E}(z=0)=0\,.
\end{equation}
Using Gr\"{o}nwall inequality and again from the uniform estimates on $\|\langle\xi\rangle^4\hat{M}^\eps_{1,1}\|$ from Lemma~\ref{lemma:M_eps_11_reg} further gives $\|\langle\xi\rangle^2\tilde{E}^\eps\|\le C|\Omega|$, which in turn gives $\|E^\eps\|\le C|\Omega|^2$. So we have
\begin{equation}\nonumber
\mathbb{E}|\upsilon^\eps(h,x;\Omega_1,\kappa)-\upsilon^\eps(h,x;\Omega_2,\kappa)|^{2n}\le C\min\{\eps^\frac13 +|\Omega|^{2n},|\Omega|^{2n}e^{c/\eta^2}\}\,.
\end{equation}
As before, it remains to use the first bound for $|\Omega|>\Omega_0=\eps^{\frac{1}{6n}}$ and the second bound for $|\Omega|\le\Omega_0$.
         \end{proof}
\end{lemma}
\subsection{Tightness in direction}
\begin{lemma}\label{lemma:tightness_direc}
  For $z=z_0+\eps\eta h$, frequency $\omega=\omega_0+\eps\eta \Omega$, lateral location $\frac{r}{\eps}+\eta x$  fixed, and $\kappa_1,\kappa_2\in\mathbb{R}^d$, with $|\kappa_1-\kappa_2|< 1$
\begin{equation}\nonumber
\mathbb{E}|\upsilon^\eps(h,x;\Omega,\kappa_1)-\upsilon^\eps(h,x;\Omega,\kappa_2)|^{2n}\le c_\alpha(n,d)|\kappa_1-\kappa_2|^{2n\alpha}\,,
    \end{equation}
    where $\alpha$ is any real number between $(0,1)$.  
    \begin{proof}
The strategy of the proof here is similar and simpler than in the previous two cases. So we outline only the key differences. Let $e^\eps=u^\eps(z,x;\omega,\sk_1)-u^\eps(z,x;\omega,\sk_2)$ with $\sk_1-\sk_2=\eps\kappa$. This follows the It\^o-Schr\"{o}dinger equation with source $e^\eps(0)=u_0^\eps(x)(e^{i\sk_1\cdot x}-e^{i\sk_2\cdot x})$. This gives for the $2n$th moment
\begin{equation}\nonumber
    \mathbb{E}|u^\eps(z,x;\omega,\sk_1)-u^\eps(z,x;\omega,\sk_2)|^{2n}:=\mu^\eps_e\,,
\end{equation}
where
\begin{equation}\nonumber
    \begin{aligned}
        \partial_z\mu^\eps_e=\frac{i\eta}{2\eps\omega}\sum_{j=1}^n(\Delta_{x_j}-\Delta_{y_j})\mu^\eps_e+\mathcal{U}_{n,n}\mu^\eps_e\\
        \mu^\eps(0)=\prod_{j=1}^nu_0(\eps x_j)u^\ast_0(\eps y_j)(e^{i\sk_1\cdot x}-e^{i\sk_2\cdot x})\,.
    \end{aligned}
\end{equation}
We then have $\|\mu^\eps_e\|_\infty\le \|\hat{\mu}^\eps_e\|$. We always have the first type of bound
\begin{equation}\nonumber
    \|\hat{\mu}^\eps_e(z)\|\le \|\hat{\mu}^\eps_e(0)\|e^{\frac{R(0)\omega^2n^2z}{\eta^2}}\le C(n)|\kappa|^{2n}e^{\frac{R(0)\omega^2n^2z}{\eta^2}}\,,
\end{equation}
where $C$ is a constant that depends on the regularity of the initial condition $\|\partial_\xi\hat{u}_0(\xi)\|$. When $|\kappa|$ is much smaller than $\eta$, we can use this. For $|\kappa|$ large, we can use the Gaussian approximation, in which case the $2n$th moment is approximated by second moments, which are bounded in the TV sense independent of $\eps$.
\end{proof}
\end{lemma}
\begin{proof}[Proof of Theorem~\ref{thm:tightness}]
From a triangle inequality, we have
\begin{equation}\nonumber
    \begin{aligned}
       & \mathbb{E}|\upsilon^\eps(h_1,x_1,\Omega_1,\kappa_1)-\upsilon^\eps(h_2,x_2,\Omega,_2\kappa_2)|^{2n}\\
       &\le C_2(n,\alpha)(\mathbb{E}|\upsilon^\eps(h_1,x_1,\Omega_1,\kappa_1)-\upsilon^\eps(h_2,x_1,\Omega_1,\kappa_1)|^{2n\alpha}+\mathbb{E}|\upsilon^\eps(h_2,x_1,\Omega_1,\kappa_1)-\upsilon^\eps(h_2,x_2,\Omega_1,\kappa_1)|^{2n\alpha}\\
        &+\mathbb{E}|\upsilon^\eps(h_2,x_2,\Omega_1,\kappa_1)-\upsilon^\eps(h_2,x_2,\Omega_2,\kappa_1)|^{2n\alpha}+\mathbb{E}|\upsilon^\eps(h_2,x_2,\Omega_2,\kappa_1)-\upsilon^\eps(h_2,x_2,\Omega_2,\kappa_2)|^{2n\alpha})
    \end{aligned}
\end{equation}
From~\cite[Theorem 2.7]{bal2024complex}, we have that for fixed $(h,\omega,\kappa)$ and $|x_1-x_2|< 1$,
\begin{equation}\nonumber
 \mathbb{E}|\upsilon^\eps(h,x_1,\Omega,\kappa)-\upsilon^\eps(h,x_2,\Omega,\kappa)|^{2n}\le C_1(n,\alpha)|x_1-x_2|^{2n\alpha}\,,
\end{equation}
for any $\alpha\in(0,1)$ and a constant $C_1$ independent of $\eps$. This, along with Lemma~\ref{lemma:tightness_z},~\ref{lemma:tightness_omega} and~\ref{lemma:tightness_direc} shows that~\eqref{eqn:ups_tight} holds. Now, choosing $n$ large enough so that $n\ge 2d+2+2n\alpha_-$ for $\alpha_-\in(0,\frac12)$ shows that $\upsilon^\eps$ has a H\"{o}lder continuous version, with H\"{o}lder exponent $\alpha_-$, and the process $(h,x,\Omega,\kappa)\mapsto\upsilon^\eps(h,x,\Omega,\kappa)$ is tight on $C^{0,\alpha_-}(\mathbb{R}^{2d+2})$~\cite[Theorem 1.4.4]{kunita1997stochastic}.%\an{this can presumably be done better with different Holder exponents for different variables, rather than viewing directly as a field on $\mathbb{R}^{2d+2}$}. 

The proof of Theorem~\ref{thm:tightness} is completed as a result of this along with Lemma~\ref{lemma:tightness_z},~\ref{lemma:tightness_omega} and~\ref{lemma:tightness_direc}.
\end{proof}

\section*{Acknowledgments} 
The authors acknowledge stimulating discussions on memory effects with Marc Guillon.
This work was funded in part by NSF grant DMS-230641.

\appendix
\section{Appendix}\label{sec:Appendix}
We will verify a modified version of Propositions 4.7 and 4.8 in~\cite{bal2024complex} that shows that the operator norm of the composition $L^{\eps,1}_{j,l}[U_\eta(U^\eps_{j',l'}-\mathbb{I})]$ is small when either $j=j'$ or $l=l'$ (but not both) while the operator norm of $L^{\eps,2}_{j,j'}$ itself is small as operators acting on $\mathcal{M}_B(\mathbb{R}^{(p+q)d})$, the Banach space of finite signed measures. 
\begin{proposition}[Modification of Proposition 4.7 in~\cite{bal2024complex}]\label{prop:L_1_bound}
Let $j=j'$ or $l=l'$, but not both. Then 
\begin{equation}\nonumber
  \sup_{0\le z'\le z}  \|\int_{z'}^zL^{\eps,1}_{j,l}(s)[U_\eta(U^\eps_{j',l'}-\mathbb{I})](s)\mathrm{d}s\|\le \frac{C_1\langle z\rangle}{\eta^4}\mathfrak{C}_1\big[\frac{C_2\eps}{\eta},\hat{\sfR},\hat{R},d\big]\,,
\end{equation}
where
\begin{equation}\nonumber
\mathfrak{C}_1[\delta,f,g,d]=C_3(\|f\|_1+\|f\|_\infty)(\|g\|_1+\|g\|_\infty)\delta\begin{cases}
        |\log\delta|^2,\quad d=1\\
        |\log\delta|,\quad d\ge 2\,.
    \end{cases}
\end{equation}
\begin{proof}
   Suppose $j=j'$, $l\neq l'$ and $\Omega_{j,l}=\Omega_l-\Omega_j$. For $\psi\in\mathcal{M}_B(\mathbb{R}^{(p+q)d})$,
   \begin{equation}\nonumber
   \begin{aligned}
             &\|\int_{z'}^zL^{\eps,1}_{j,l}(s)[U_\eta(U^\eps_{j',l'}-\mathbb{I})](s)\mathrm{d}s\psi\|\\
             &=\frac{\omega_j^2\omega'_l\omega'_{l'}}{(4\eta^2)^2}\|\int_{\mathbb{R}^{2d}}\int_{z'}^z\int_{0}^{s_1}U_\eta(s_1)U^\eps_{j,l'}(s_2,\xi_j-k_1-k_2,\zeta_l-k_1,\zeta_{l'}-k_2)\psi(\xi_j-k_1-k_2,\zeta_l-k_1,\zeta_{l'}-k_2)\\
             &e^{\frac{i\eta s_1}{2\eps}\big[-\frac{|k_1|^2\eps\eta\Omega_{j,l}}{\omega_j\omega'_l}+2k_1\cdot\big(\frac{\xi_j}{\omega_j}-\frac{\zeta_l}{\omega'_l}\big)\big]}e^{\frac{is_2\eta}{2\eps}\big[-\frac{|k_2|^2\eps\eta\Omega_{j,l'}}{\omega_j\omega'_{l'}}+2k_2\cdot\big(\frac{\xi_j-k_1}{\omega_j}-\frac{\zeta_{l'}}{\omega'_{l'}}\big)\big]}\frac{\hat{R}(k_1)\hat{R}(k_2)\mathrm{d}k_1\mathrm{d}k_2\mathrm{d}s_2\mathrm{d}s_1}{(2\pi)^{2d}}\|\,.
   \end{aligned}
   \end{equation}
   From the change of variables $\xi_j-k_1-k_2\to \xi_j, \zeta_l-k_1\to \zeta_l, \zeta_{l'}-k_2\to\zeta_{l'}$, 
      \begin{equation}\nonumber
   \begin{aligned}
             &\|\int_{z'}^zL^{\eps,1}_{j,l}(s)[U_\eta(U^\eps_{j',l'}-\mathbb{I})](s)\mathrm{d}s\psi\|\\
            =&\frac{\omega_j^2\omega'_l\omega'_{l'}}{(4\eta^2)^2}\|\int_{\mathbb{R}^{2d}}\int_{z'}^z\int_{0}^{s_1}U_\eta(s_1)U^\eps_{j,l'}(s_2,v)\psi(v)e^{\frac{i\eta s_1}{\eps}k_1\cdot\big(\frac{\xi_j}{\omega_j}-\frac{\zeta_l}{\omega'_l}+\frac{k_2}{\omega_j}\big)}e^{\frac{i\eta s_2}{\eps}k_2\cdot\big(\frac{\xi_j}{\omega_j}-\frac{\zeta_{l'}}{\omega'_{l'}}\big)}\\
             &e^{\frac{i\eta^2}{2}\big[\frac{s_1|k_1|^2\Omega_{j,l}}{\omega_j\omega'_l}+\frac{s_2|k_2|^2\Omega_{j,l'}}{\omega_j\omega'_{l'}}\big]}\frac{\hat{R}(k_1)\hat{R}(k_2)\mathrm{d}k_1\mathrm{d}k_2\mathrm{d}s_2\mathrm{d}s_1}{(2\pi)^{2d}}\|\\
             =& \frac{\omega_j^2\omega'_l\omega'_{l'}}{(4\eta^2)^2}\|\int_{\mathbb{R}^{2d}}\frac{\hat{R}(k_1)\hat{R}(k_2)\mathrm{d}k_1\mathrm{d}k_2\mathrm{d}s_1\mathrm{d}s_2}{(2\pi)^{2d}}\\
             &\Big(\int_{0}^{z'}\int_{z'}^{z}+\int_{z'}^z\int_{s_2}^{z}\Big)U_\eta(s_1)U^\eps_{j,l'}(s_2,v)\psi(v)e^{\frac{i\eta s_1}{\eps}k_1\cdot\big(\frac{\xi_j}{\omega_j}-\frac{\zeta_l}{\omega'_l}+\frac{k_2}{\omega_j}\big)}e^{\frac{i\eta s_2}{\eps}k_2\cdot\big(\frac{\xi_j}{\omega_j}-\frac{\zeta_{l'}}{\omega'_{l'}}\big)}e^{\frac{i\eta^2}{2}\big[\frac{s_1|k_1|^2\Omega_{j,l}}{\omega_j\omega'_l}+\frac{s_2|k_2|^2\Omega_{j,l'}}{\omega_j\omega'_{l'}}\big]}\|\\
             \le& \frac{\omega_j^2\omega'_l\omega'_{l'}}{(4\eta^2)^2}\Big(\big\|\int_{\mathbb{R}^{2d}}\frac{\hat{R}(k_1)\hat{R}(k_2)\mathrm{d}k_1\mathrm{d}k_2}{(2\pi)^{2d}}\int_0^{z'}|U^\eps_{j,l'}(s_2)\psi|\mathrm{d}s_2\big|\int_{z'}^zU_\eta(s_1)e^{\frac{i\eta s_1}{\eps\omega_j}k_1\cdot\big(\xi_j-\frac{\zeta_l\omega_j}{\omega'_l}+\frac{\eta\eps\Omega_{j,l}k_1}{2\omega'_l}+k_2\big)}\mathrm{d}s_1\big|\big\| \\
             &+\big\|\int_{\mathbb{R}^{2d}}\frac{\hat{R}(k_1)\hat{R}(k_2)\mathrm{d}k_1\mathrm{d}k_2}{(2\pi)^{2d}}\int_{z'}^z|U^\eps_{j,l'}(s_2)\psi|\mathrm{d}s_2\big|\int_{s_2}^zU_\eta(s_1)e^{\frac{i\eta s_1}{\eps\omega_j}k_1\cdot\big(\xi_j-\frac{\zeta_l\omega_j}{\omega'_l}+\frac{\eta\eps\Omega_{j,l}k_1}{2\omega'_l}+k_2\big)}\mathrm{d}s_1\big|\big\|\Big)\\
             \le& \frac{\omega_j^2\omega'_l\omega'_{l'}}{(4\eta^2)^2}\sup_{w\in\mathbb{R}^d}\Big(\int_{\mathbb{R}^{2d}}\frac{\hat{R}(k_1)\hat{R}(k_2)\mathrm{d}k_1\mathrm{d}k_2}{(2\pi)^{2d}}\int_0^{z'}\|U^\eps_{j,l'}(s_2)\psi\|\mathrm{d}s_2\big|\int_{z'}^zU_\eta(s_1)e^{\frac{i\eta s_1}{\eps\omega_j}k_1\cdot(k_2+w)}\mathrm{d}s_1\big| \\
             &+\sup_{w\in\mathbb{R}^d}\int_{\mathbb{R}^{2d}}\frac{\hat{R}(k_1)\hat{R}(k_2)\mathrm{d}k_1\mathrm{d}k_2}{(2\pi)^{2d}}\int_{z'}^z\|U^\eps_{j,l'}(s_2)\psi\|\mathrm{d}s_2\big|\int_{s_2}^zU_\eta(s_1)e^{\frac{i\eta s_1}{\eps\omega_j}k_1\cdot(k_2+w)}\mathrm{d}s_1\big|\Big)\,.
   \end{aligned}
   \end{equation}
   The rest of the proof can be completed as in the proof of~\cite[Proposition 4.7]{bal2024complex}.
\end{proof}
\end{proposition}
\begin{proposition}[Modification of Proposition 4.8 in~\cite{bal2024complex}]\label{prop:L_2_bound}
    \begin{equation}\nonumber
      \sup_{0\le z'\le z}  \|\int_{z'}^zL^{\eps,2}_{j,j'}(s)\mathrm{d}s\|\le \frac{c\langle z\rangle}{\eta^2}\mathfrak{C_2}\big[\frac{\omega_j\omega_{j'}}{\Bar{\omega}_{j,j'}}\frac{\eps}{\eta},\hat{R},d\big],\quad 1\le j<j'\le p\,,
    \end{equation}
    where 
    \begin{equation}\nonumber
        \mathfrak{C}_2[\delta,f,d]=\begin{cases}
            2\pi\|f\|_\infty\sqrt{\delta},\quad d=1\\
            C(\|\langle\xi\rangle^{d-2}f(\xi)\|_\infty\delta|\log\delta|+\|f\|_1\delta),\quad d\ge 2\,.
        \end{cases}
    \end{equation}
    Similar bounds hold for $L^{\eps,2}_{l,l'}, p+1\le l<l'\le p+q$, with $\omega_j,\omega_{j'}$ replaced by $\omega'_l, \omega'_{l'}$.
    \begin{proof}
       For $\psi\in\mathcal{M}_B(\mathbb{R}^{(p+q)d})$,
        \begin{equation}\nonumber
            \|\int_{z'}^zL^{\eps,2}_{j,j'}(s)\psi(s)\mathrm{d}s\|=\frac{\omega_j\omega_{j'}}{4\eta^2}\|\int_{z'}^z\int_{\mathbb{R}^d}\psi(\xi_j-k,\xi_{j'}+k)e^{\frac{i\eta s}{\eps}\big[-\frac{|k|^2\Bar{\omega}_{j,j'}}{\omega_j\omega_{j'}}+k\cdot\big(\frac{\xi_j}{\omega_j}-\frac{\xi_{j'}}{\omega_{j'}}\big)\big]}\frac{\hat{R}(k)\mathrm{d}k\mathrm{d}s}{(2\pi)^d}\|\,. 
        \end{equation}
        From the change of variables $\xi_j-k\to \xi_j, \xi_{j'}+k\to \xi_{j'}$,
                \begin{equation}\nonumber
                \begin{aligned}
       \|\int_{z'}^zL^{\eps,2}_{j,j'}(s)\psi(s)\mathrm{d}s\|&=\frac{\omega_j\omega_{j'}}{4\eta^2}\|\int_{z'}^z\int_{\mathbb{R}^d}\psi(v)e^{\frac{i\eta s}{\eps}\big[\frac{|k|^2\Bar{\omega}_{j,j'}}{\omega_j\omega_{j'}}+k\cdot\big(\frac{\xi_j}{\omega_j}-\frac{\xi_{j'}}{\omega_{j'}}\big)\big]}\frac{\hat{R}(k)\mathrm{d}k\mathrm{d}s}{(2\pi)^d}\|\\        
       &\le \frac{\omega_j\omega_{j'}}{4\eta^2}\|\psi\|\sup_{w\in\mathbb{R}^d}\int_{\mathbb{R}^d}\frac{\hat{R}(k)\mathrm{d}k}{(2\pi)^d}\Big|\int_{z'}^ze^{\frac{is\eta \Bar{\omega}_{j,j'} }{\eps\omega_j\omega_{j'}}k\cdot(k+w)}\mathrm{d}s\Big|\,.
                \end{aligned}
                 \end{equation}
                 The rest of the proof follows as in~\cite[Proposition 4.8]{bal2024complex}.
    \end{proof}
\end{proposition}
\bibliographystyle{siam}
\bibliography{Reference}

\end{document}